\documentclass[11pt,letterpaper,oneside,english]{amsart}
\usepackage[utf8]{inputenc}
\usepackage{amsbsy}
\usepackage{amstext}
\usepackage{amsthm}
\usepackage{amssymb}
\usepackage{graphicx}

\makeatletter

\numberwithin{equation}{section}
\numberwithin{figure}{section}
\theoremstyle{plain}
\newtheorem{thm}{\protect\theoremname}[section]
\theoremstyle{remark}
\newtheorem{rem}[thm]{\protect\remarkname}
\theoremstyle{definition}
\newtheorem{example}[thm]{\protect\examplename}
\theoremstyle{plain}
\newtheorem{lem}[thm]{\protect\lemmaname}
\theoremstyle{plain}
\newtheorem{cor}[thm]{\protect\corollaryname}
\theoremstyle{definition}
\newtheorem{defn}[thm]{\protect\definitionname}
\theoremstyle{plain}
\newtheorem{prop}[thm]{\protect\propositionname}

\include{mypackages}
\include{mymacros}
\usepackage{appendix}
\usepackage{chngcntr}
\usepackage{apptools}
\usepackage{indentfirst}
\usepackage{graphicx}
\usepackage{amssymb}
\usepackage{caption}
\usepackage{subcaption}
\usepackage{indentfirst}

\renewcommand{\d}{\partial}

\def\R{\mathbb R}

\def\<{\langle}
\def\>{\rangle}

\makeatother

\usepackage{babel}
\providecommand{\corollaryname}{Corollary}
\providecommand{\definitionname}{Definition}
\providecommand{\examplename}{Example}
\providecommand{\lemmaname}{Lemma}
\providecommand{\propositionname}{Proposition}
\providecommand{\remarkname}{Remark}
\providecommand{\theoremname}{Theorem}

\begin{document}
\title[Curve shortening flow ]{Curve shortening flow on Riemann surfaces with possible ambient conic
singularities }
\author{Biao Ma}
\address{14 MacLean Hall, University of Iowa, IA, 52242}
\email{biao-ma@uiowa.edu}
\begin{abstract}
In this paper, we study the curve shortening flow (CSF) on Riemann
surfaces. We generalize Huisken's comparison function to Riemann surfaces
and surfaces with conic singularities. We reprove the Gage-Hamilton-Grayson
theorem on surfaces. We also prove that for embedded simple closed
curves, CSF can not touch conic singularities with cone angles smaller
than or equal to $\pi$. 
\end{abstract}

\maketitle

\section{Introduction}

In this paper, we study curve shortening flow (CSF) on surfaces. Let
$(M,g)$ be a compact Riemann surface. We consider a one-parameter
family of simple closed curves $\{\Gamma_{t}\}_{t\in[0,T)}$ on $M$.
Suppose that $\Gamma_{t}$ is parametrized by 
\begin{align*}
X & :S^{1}\times[0,T)\to M.
\end{align*}
We call $\{\Gamma_{t}\}$ a \textbf{curve shortening flow} on $M$
if 
\[
\frac{\partial}{\partial t}X(u,t)=k\boldsymbol{n},
\]
where $k$ is the geodesic curvature and $\boldsymbol{n}$ the normal
vector at $X(u,t)$. We only consider CSF for embedded simple closed
curves.

The classic Gage-Hamilton-Grayson Theorem states that: A smooth embedded
closed curve will stay smooth and embedded under CSF unless it converges
to a single round point. See \cite{gage1986heat,grayson1987heat,grayson1989shortening,gage1990curve}.
Here a round point means by a rescaling procedure, the limiting curve
is a round circle. Huisken gives an alternative proof for planar curves
\cite{huisken1998distance}. His proof utilizes a monotonic comparison
function aimed at the isoperimetric profile on $\mathbb{R}^{2}$.

In this paper, we generalize Huisken's proof to surfaces and give
an alternative proof of Gage-Hamilton-Grayson theorem. A key construction
is the following function which is inspired by Huiksen's comparison
function \cite{huisken1998distance}:
\begin{equation}
R(t)=\sup_{x,y}\frac{L(t)e^{-Kt}}{\pi d(x,y,t)}\sin\left(\frac{\pi l(x,y,t)}{L(t)}\right),\label{eq:R(t) definition}
\end{equation}
where $d(x,y,t)$ is the distance between $X(x,t)$ and $X(y,t)$,
$L(t)$ is the total length of $X(\cdot,t)$, and $l(x,y,t)$ is the
arc-length between $x$ and $y$. We state our first main result.
\begin{thm}
Suppose that $X:S^{1}\times[0,T)\to M$ is a curve shortening flow.
Let $R(t)$ be the comparison function defined in (\ref{eq:R(t) definition}).
Let $d_{M}$ be the injectivity radius and $K_{M}=\sup_{M}|\mathcal{K}|$
be the Gaussian curvature bound. Suppose that the curve shortening
flow exists in the time period $[0,T)$. There exists a constant $K=K(d_{M},K_{M})$
such that $R(t)$ is bounded by a constant $C(R(0),d_{M},K_{M})$
.\label{thm:Suppose-that-first theorem}
\end{thm}

\begin{rem}
One can not expect $R(t)$ to be monotonic. In fact, we can only perform
variation of distance function inside a small neighborhood. So the
bound of $R(t)$ can be huge. However, the boundedness of $R(t)$
is enough to rule out Type II singularities. See Section 2.  Brendle
\cite{brendle2014two} gives a nice exposition for some generalizations
of Huisken's comparison functions.
\end{rem}

Together with an analysis on Type I singularities in Section 5, we
can prove Gage-Hamilton-Grayson Theorem:
\begin{thm}
\label{thm:G-H-G thm}Suppose that $\Gamma_{t}$ is a curve shortening
flow starting with a smooth embedded curve. Then, $\Gamma_{t}$ remains
smooth and embedded unless it converges to a round point.
\end{thm}

\begin{rem}
Our approach follows from Huisken's proof \cite{huisken1998distance}
on planar curves which is summarized in Section \ref{sec:Backgrounds}.
Andrews and Bryan \cite{andrews2011curvature} also give a proof for
planar CSF by constructing a very delicate comparison function. 
\end{rem}

It is worth noting that our comparison function can be applied to
CSF on conic Riemann surfaces. Conic Riemann surfaces are surfaces
with cone-like singularities. The simplest example is the \textbf{\emph{flat
cone}} 
\begin{equation}
C_{a\pi}=(\mathbb{C},g_{a\pi}),g_{a\pi}=|z|^{a-2}|dz|^{2},\label{eq:flat cone}
\end{equation}
for some $0<a$. It has a conic singularity at the origin with cone
angle $a\pi$ if $0<a\pi<2\pi$. 

Now, imagine curve shortening flow on $C_{\alpha\pi}$ as a shrinking
rubber band. See Figure \ref{Rubberband}. 
\begin{figure}

\centering
\begin{subfigure}{5.5cm}
\includegraphics[width=3.2cm]{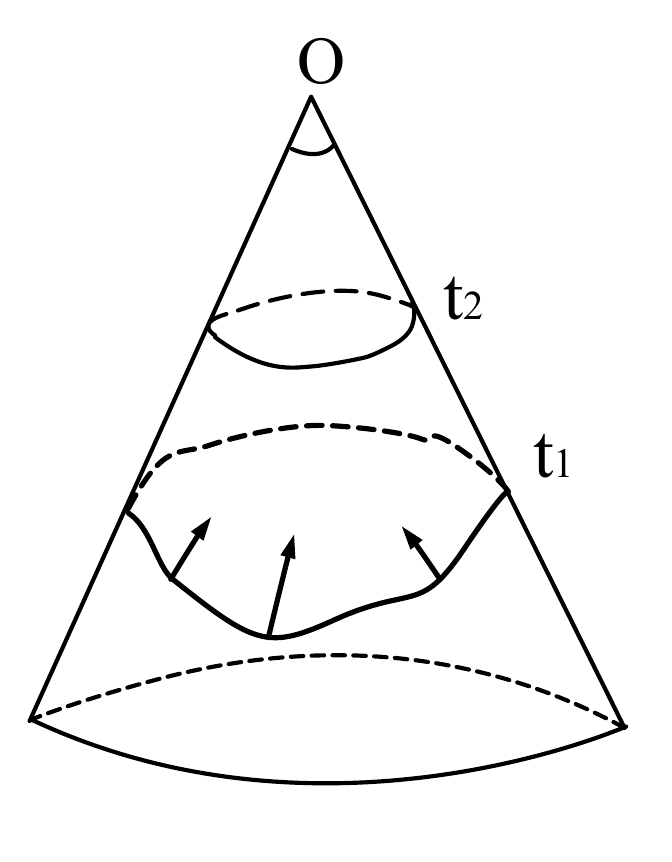}

\end{subfigure}
\begin{subfigure}{5cm}
\includegraphics[width=4.8cm]{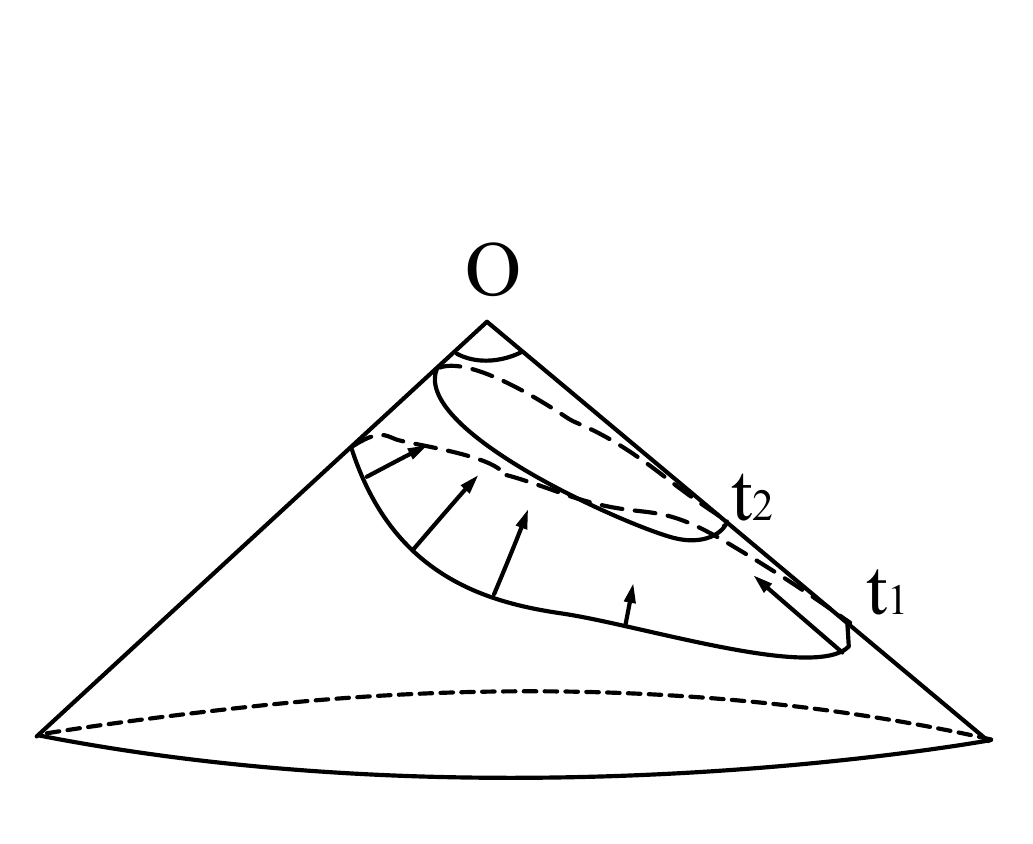}

\end{subfigure}

\caption{CSF on flat cones}
\label{Rubberband}
\end{figure}
There are two cases: If the cone angle is big, the cone tip is not
a big obstacle and the rubber band can move across the cone tip to
another side. In the extreme case $\alpha=2$, $C_{2\pi}$ is just
$\R^{2}$ with a fake cone tip at the origin, and some CSF can certainly
pass the origin. However, if the cone angle is very sharp, the rubber
band can not easily touch the cone tip. It shrinks uniformly until
the whole rubber band goes to the cone tip. Between the two types
of cones, the threshold is $C_{\pi}$. For $C_{\pi}$, we consider
the double cover of $C_{\pi}$ which is exactly $\mathbb{C}=\mathbb{R}^{2}$.
We may thus lift the CSF to $\mathbb{R}^{2}$. See Figure \ref{C_=00005Cpi}.
\begin{figure}
\centering\includegraphics[scale=0.7]{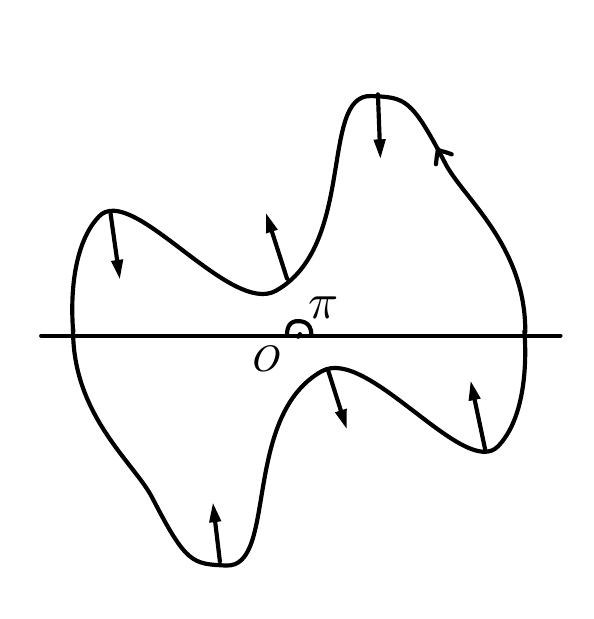}

\caption{CSF on a double cover of $C_{\pi}$}
\label{C_=00005Cpi}
\end{figure}
The maximum principle of parabolic equations shows that the CSF remains
embedded and hence by the symmetry of the lifting, the CSF can not
touch the cone tip in $C_{\pi}$. 

For conic Riemann surfaces, we prove that CSF can not touch the cone
tip if the cone angle is between $0$ and $\pi$. 
\begin{thm}
Let $M$ be a conic Riemann surface. Suppose that $\{\Gamma_{t}\}_{t\in[0,T)}$
is a curve shortening flow starting with a smooth closed embedded
curve $\Gamma_{0}$. Assume $T<\infty$. If $p\in M$ is a conic singularity
with cone angle less than or equal to $\pi$, then either $\lim_{t\to T}\Gamma_{t}=p,$
or the distance function $\mathbf{d}(\Gamma_{t},p)>c>0$ for all $t\in(0,T)$.\label{thm-second them}
\end{thm}

The conic singularities on surfaces are ambient singularities of CSF.
By Theorem \ref{thm-second them}, CSF in some sense can detect ambient
singularities. This is quite different from $\mathbb{R}^{2}$ or smooth
surfaces. In $\mathbb{R}^{2}$, a priori, it is not easy to forecast
the precise ending point of a CSF. We intend to study the mean curvature
flow with ambient singularities in higher dimensions. It is reasonable
to expect mean curvature flow to detect ambient singularities in a
similar fashion. 

We organize the paper as follows: In Section \ref{sec:Backgrounds},
we provide some background and historic notes of CSF. We sketch the
idea of Huisken's proof of Gage-Hamilton-Grayson Theorem in $\mathbb{R}^{2}$.
In Section 3, we collect some formulae for the distance functions
on surfaces. In Section 4, we prove Theorem \ref{thm:Suppose-that-first theorem}.
In Section 5, we study Type I singularities on surfaces. In Section
\ref{sec:conic surfaces}, we provide an introduction on conic Riemann
surfaces and discuss the construction of a double branched cover.
In Section 7, we prove Theorem \ref{thm-second them}.

\subsection*{Acknowledgments }

I would like to thank Jian Song for introducing me this topic. I want
to thank Hao Fang, Jian Song, and James Dibble for discussions and
comments. I thank Marc Troyanov for pointing out a flaw in my original
proof. I want to thank Shanghai Jiaotong University and Mijia Lai
for hospitality during my visit there. 

\section{Backgrounds and historical notes\label{sec:Backgrounds}}

Let $(M,g)$ be a compact Riemann surface. The curve shortening flow
$\{\Gamma_{t}\}$ is parametrized by: 
\begin{align*}
X & :S^{1}\times[0,T)\to M,
\end{align*}
where $X$ evolves under the governing equation: 
\[
\frac{\partial}{\partial t}X(u,t)=k\boldsymbol{n}.
\]
Here $k$ is the geodesic curvature and $\boldsymbol{n}$ is the normal
vector at $X(u,t)$ such that the orientation $\{\frac{\partial}{\partial u}X,\boldsymbol{n}\}$
agrees with the preferred orientation on $M$. The total length $L(t)$
of $\Gamma_{t}$ is decreasing along the CSF \cite{grayson1989shortening}.
In fact, let $s$ be the length parameter such that $L(t)=\int_{\Gamma_{t}}ds$.
Then,
\begin{align}
\frac{d}{dt}ds & =\frac{\partial}{\partial t}\langle\frac{\partial X}{\partial u},\frac{\partial X}{\partial u}\rangle^{\frac{1}{2}}du=\|\frac{\partial X}{\partial u}\|^{-1}\langle\frac{\partial^{2}X}{\partial u\partial t},\frac{\partial X}{\partial u}\rangle du\label{eq:L(t) decay}\\
 & =\langle\frac{\partial^{2}X}{\partial t\partial u},\frac{\partial X}{\partial s}\rangle du=\langle\frac{\partial}{\partial u}(k\boldsymbol{n}),\frac{\partial X}{\partial s}\rangle du\nonumber \\
 & =\left(\frac{\partial}{\partial s}\langle k\boldsymbol{n},\frac{\partial X}{\partial s}\rangle\right)ds-\langle k\boldsymbol{n},\frac{\partial^{2}X}{\partial s^{2}}\rangle ds\nonumber \\
 & =-\langle k\boldsymbol{n},k\boldsymbol{n}\rangle ds=-k^{2}ds.\nonumber 
\end{align}
Hence, 
\[
\frac{d}{dt}L(t)=-\int_{\Gamma_{t}}k^{2}ds.
\]
Therefore, CSF can be view as a gradient flow of the total length
functional.

Next, we recall two exact solutions of CSF.
\begin{example}
(Shrinking round circle). In $\mathbb{R}^{2}$, suppose that $\Gamma_{t}$
are circles of radius $R(t)=\sqrt{R_{0}^{2}-2t}$ with center at $0$
. Then $\{\Gamma_{t}\}$ is a CSF for $t\in(-\infty,R_{0}^{2}/2)$.
Clearly, $\{\Gamma_{t}\}$ is a self-similar shrinking solution.
\end{example}

\begin{example}
(Grim Reaper). Let $\Gamma_{t}={\rm graph\ of}\{-\log\cos x\}+t\subset\mathbb{R}^{2}$
for $x\in(-\pi/2,\pi/2)$ and $t\in(-\infty,\infty)$. It is a translating
self-similar solution 
\end{example}

Historically, Gage and Hamilton \cite{gage1986heat} first obtained
the shrinking results of CSF for convex planar curves. They proved
that under CSF, closed embedded convex curves in $\mathbb{R}^{2}$
shrink to a round point. Grayson in\cite{grayson1987heat} proved
the same results for general curves without requiring convexity. In
fact, he proved that any closed embedded curve will eventually become
convex and shrink to a point. Grayson's proof was quite delicate in
which he combined different types of analysis to deal with various
geometric configurations.

CSF is a special case of general \textbf{mean curvature flow} (MCF),
which by its name means embedded hypersurfaces evolving by their mean
curvature vectors. It has short time existence as long as the geodesic
curvature stays uniformly bounded. Therefore, CSF exists and keeps
embedded in a time interval $[0,T)$. When the geodesic curvature
$k\to\infty$ as $t\to T$, singularities will happen.

A different proof of Gage-Hamilton-Grayson for planar curves is based
on analyzing the behavior at the first time singularities \cite{huisken1990asymptotic,huisken1998distance,altschuler1991singularities,angenent1991formation}.
See \cite{haslhofer2016lectures} for an excellent exposition. Let
$T$ be the time when the first singularity happens. The singularities
of CSF can be classified into 2 types :
\begin{itemize}
\item Type I singularities: $\sup_{\Gamma_{t}}|k|\leq\frac{C}{\sqrt{T-t}}$.
\item Type II singularities: $\limsup_{t\to T}\left((T-t)\sup_{\Gamma_{t}}|k|^{2}\right)=\infty.$
\end{itemize}
To deal with Type I singularities, Huisken defined the backward heat
kernel \cite{huisken1990asymptotic}: 
\[
\rho_{x_{0}}(x,t)=\left(4\pi\left(t_{0}-t\right)\right)^{-1/2}e^{-\frac{\left|x-x_{0}\right|^{2}}{4\left(t_{0}-t\right)}}.\mathbf{}
\]
He showed

\begin{equation}
\frac{d}{dt}\int_{\Gamma_{t}}\rho_{x_{0}}ds=-\int_{\Gamma_{t}}\left|k+\frac{\langle X,\boldsymbol{n}\rangle}{2\left(t_{0}-t\right)}\right|^{2}\rho_{x_{0}}ds<0.\label{eq:Backwards heat kernel}
\end{equation}
After a rescaling, Type I singularities converge to some ancient CSF
and by (\ref{eq:Backwards heat kernel}), they happen to be self-similar
shrinking solutions of CSF, which can only be round circles\cite{abresch1986normalized,angenent1991formation}.
Similar results can be derived for MCF on $\mathbb{R}^{n}$ \cite{huisken1990asymptotic},
except for MCF, more self-similar shrinking solutions exist. Hamilton
\cite{hamilton1993monotonicity} generalized Huisken's formula to
manifolds assuming some additional geometric conditions. We will establish
a localized backward heat kernel formula on surfaces to study Type
I singularities in Section \ref{sec:Type-I}.

A big difference between CSF and higher dimensional MCF is the non-existence
of Type II singularities. For CSF of a simple closed embedded curve
on $\mathbb{R}^{2}$, Type II singularities can never happen while
they can happen in general for MCF. The absence of Type II singularities
of CSF on $\mathbb{R}^{2}$ can be seen using Huisken's comparison
function:
\[
R_{H}(t):=\sup_{x,y}\frac{L(t)}{\pi d(x,y,t)}\sin\left(\frac{\pi l(x,y,t)}{L(t)}\right).
\]
Huisken proved that for embedded curves $R_{H}(t)$ is non-increasing
and has lower bound $1$ which can be achieved only by round circles
\cite{huisken1998distance}. To rule out Type II singularities, we
can perform Type II blow-up by Altschuler \cite{altschuler1991singularities}
and the limit converges to a grim reaper by Hamilton's Harnack estimate
\cite{hamilton1995harnack}. However, grim reaper apparently has an
unbounded ratio between intrinsic and extrinsic distance which implies
that $R_{H}(t)$ is unbounded. This contradicts the monotonicity of
$R_{H}(t)$. 

Once Type II singularities are excluded, it is clear that the first
time singularity can only be Type I and hence a round point. Therefore,
the Gage-Hamilton-Grayson theorem on $\mathbb{R}^{2}$ is proved.
See details in \cite{huisken1998distance}.

\section{Distance Variation Formulae.}

In this section, we present the distance variation formulae along
a geodesic connecting two points on a given curve shortening flow.
Many of the formulae have been derived in \cite{gage1990curve,grayson1989shortening}.
We collect them for completeness. 

Let $X(u,t)$ be a CSF on a Riemann surface as in Section 2. Let $\gamma$
be the shortest geodesic connecting $X(u_{1},t)$ and $X(u_{2},t)$.
Let $d$ be the length of $\gamma$. Parametrize $\gamma(\alpha)$
such that
\[
\dot{\gamma}=\frac{d}{d\alpha}\gamma,|\dot{\gamma}|=d.
\]
By this parametrization, $\gamma(0)=X(u_{1},t)$ and $\gamma(1)=X(u_{2},t)$.
Pick a parallel orthonormal frame $\boldsymbol{\tau},\boldsymbol{\nu}$
along the geodesic such that 
\[
\boldsymbol{\tau}=\frac{\dot{\gamma}}{|\dot{\gamma}|},\boldsymbol{\nu}\perp\boldsymbol{\tau},
\]
and $(\boldsymbol{\tau},\boldsymbol{\nu})$ agrees with the preferred
orientation on the surface. We will use $\dot{\gamma}_{i},k_{i},\boldsymbol{n}_{i},\boldsymbol{\tau}_{i},\boldsymbol{\nu}_{i}$
respectively to represent $\dot{\gamma},k,\boldsymbol{n},\boldsymbol{\tau},\boldsymbol{\nu}$
evaluated at $X(u_{i},t)$, for $i=1,2$. Let $U_{i}(\alpha)$, $i=1,2$
be the Jacobi field along $\gamma$ such that
\[
\begin{cases}
U_{1}(0)=\frac{\partial}{\partial u_{1}}X(u_{1},t),\\
U_{1}(1)=0,
\end{cases}\quad\begin{cases}
U_{2}(0)=0,\\
U_{2}(1)=\frac{\partial}{\partial u_{2}}X(u_{2},t).
\end{cases}
\]
Let $\boldsymbol{T}_{i}$ be the unit tangent vector at $X(u_{i},t)$.
Extend $\boldsymbol{T}_{i}$ to a Jacobi field $\tilde{\boldsymbol{T}_{i}}$
such that $\tilde{\boldsymbol{T}}_{i}=\boldsymbol{T}_{i}$ at $X(u_{i},t)$
and equals $0$ at the other end. Let $W(\alpha)$ be the Jacobi field
along $\gamma$ such that 
\[
W(0)=k_{1}\boldsymbol{n}_{1},W(1)=k_{2}\boldsymbol{n}_{2}.
\]
Let $F(u_{1},u_{2},t)=\int_{0}^{1}|\dot{\gamma}|d\alpha$ be the distance. 

\begin{align*}
F_{t} & =\int_{0}^{1}|\dot{\gamma}|_{t}d\alpha=\int_{0}^{1}\frac{\langle\dot{\gamma},\nabla_{W}\dot{\gamma}\rangle}{|\dot{\gamma}|}d\alpha\\
 & =\frac{1}{d}\int_{0}^{1}\langle\dot{\gamma},\nabla_{\dot{\gamma}}W\rangle d\alpha=\frac{1}{d}\int_{0}^{1}\frac{\partial}{\partial\alpha}\langle\dot{\gamma},W\rangle d\alpha\\
 & =\frac{1}{d}\langle\dot{\gamma}_{2},k_{2}\boldsymbol{n}_{2}\rangle-\frac{1}{d}\langle\dot{\gamma}_{1},k_{1}\boldsymbol{n}_{1}\rangle.
\end{align*}
Similarly,
\[
F_{u_{1}}=\frac{1}{d}\int_{0}^{1}\text{\ensuremath{\frac{\partial}{\partial\alpha}\langle\dot{\gamma},U_{1}\rangle d\alpha=-\frac{1}{d}\langle\dot{\gamma}_{1},U_{1}\rangle,F_{u_{2}}=\frac{1}{d}\langle\dot{\gamma}_{2},U_{2}\rangle.}}
\]
With arc-length parameterization, we have 
\begin{equation}
F_{s_{1}}=-\frac{1}{d}\langle\dot{\gamma}_{1},\boldsymbol{T}_{1}\rangle|_{X(u_{1},t)},F_{s_{2}}=\frac{1}{d}\langle\dot{\gamma}_{2},\boldsymbol{T}_{2}\rangle|_{X(u_{2},t)}.\label{eq:F_s_1}
\end{equation}
Calculate the second order derivatives:

\begin{align*}
(F_{s_{1}})_{u_{1}} & =-\frac{\partial}{\partial u_{1}}(\frac{1}{d}\langle\dot{\gamma}_{1},\boldsymbol{T}_{1}\rangle)\\
 & =-\frac{|U_{1}(0)|\langle\dot{\gamma}_{1},\boldsymbol{T}_{1}\rangle^{2}}{d^{3}}-\frac{1}{d}\langle\nabla_{U_{1}}\dot{\gamma}_{1},\boldsymbol{T}_{1}\rangle-\frac{1}{d}\langle\dot{\gamma}_{1},\nabla_{U_{1}}\boldsymbol{T}_{1}\rangle\\
 & =\frac{|U_{1}(0)|}{d}\left(-\frac{\langle\dot{\gamma}_{1},\boldsymbol{T}_{1}\rangle^{2}}{d^{2}}-\langle\nabla_{\dot{\gamma}_{1}}\tilde{\boldsymbol{T}}_{1},\boldsymbol{T}_{1}\rangle-\langle\dot{\gamma}_{1},\nabla_{\tilde{\boldsymbol{T}}_{1}}\boldsymbol{T}_{1}\rangle\right).
\end{align*}
Let 
\[
\mathcal{K}:M\to\mathbb{R}
\]
be the Gaussian curvature of $M$. Let $J_{1}(\alpha),J_{2}(\alpha)$
be the solutions of 
\begin{align}
J_{1}(1) & =0,J_{1}'(1)=-d,J_{1}''(\alpha)+d^{2}\mathcal{K}(\gamma(\alpha))J_{1}(\alpha)=0,\label{eq:Jacobian equations}\\
J_{2}(0) & =0,J_{2}'(0)=-d,J_{2}''(\alpha)+d^{2}\mathcal{K}(\gamma(\alpha))J_{2}(\alpha)=0.\nonumber 
\end{align}
Then, we can write down the Jacobi field $\tilde{\boldsymbol{T}}_{1}$
and its derivative:
\begin{align}
\tilde{\boldsymbol{T}}_{1}(\alpha) & =\langle\boldsymbol{T}_{1},\boldsymbol{\tau}_{1}\rangle(1-\alpha)\boldsymbol{\tau}+\langle\boldsymbol{T}_{1},\boldsymbol{\nu}_{1}\rangle\frac{J_{1}(\alpha)}{J_{1}(0)}\boldsymbol{\nu},\label{eq:Jacobian vector fields}\\
\nabla_{\dot{\gamma}}\tilde{\boldsymbol{T}}_{1}|_{\alpha=0} & =-\langle\boldsymbol{T}_{1},\boldsymbol{\tau}_{1}\rangle\boldsymbol{\tau}_{1}+\langle\boldsymbol{T}_{1},\boldsymbol{\nu}_{1}\rangle\frac{J'_{1}(0)}{J_{1}(0)}\boldsymbol{\nu}_{1}.\nonumber 
\end{align}
Hence, the second variation formula shows
\begin{align*}
F_{s_{1}s_{1}} & =\frac{1}{d}\left(-\frac{\langle\dot{\gamma}_{1},\boldsymbol{T}_{1}\rangle^{2}}{d^{2}}+\langle\boldsymbol{T}_{1},\boldsymbol{\tau}_{1}\rangle^{2}-\frac{J_{1}'(0)}{J_{1}(0)}\langle\boldsymbol{T}_{1},\boldsymbol{\nu}_{1}\rangle^{2}-\langle\dot{\gamma}_{1},k_{1}\boldsymbol{n}_{1}\rangle\right)\\
 & =-\frac{1}{d}\frac{J_{1}'(0)}{J_{1}(0)}\langle\boldsymbol{T}_{1},\boldsymbol{\nu}_{1}\rangle^{2}-\langle\frac{\dot{\gamma}_{1}}{d},k_{1}\boldsymbol{n}_{1}\rangle.
\end{align*}
Note $\boldsymbol{T}_{1}=\langle\boldsymbol{T}_{1},\boldsymbol{\tau}_{1}\rangle\boldsymbol{\tau}_{1}+\langle\boldsymbol{T}_{1},\boldsymbol{\nu}_{1}\rangle\boldsymbol{\nu}_{1}.$
By parallel transport, we can extend $\boldsymbol{T}_{1}$ such that
\[
\bar{\boldsymbol{T}}_{1}(\alpha)=\langle\boldsymbol{T}_{1},\boldsymbol{\tau}_{1}\rangle\boldsymbol{\tau}(\alpha)+\langle\boldsymbol{T}_{1},\boldsymbol{\nu}_{1}\rangle\boldsymbol{\nu}(\alpha).
\]
Notice that 
\begin{align*}
\langle\dot{\gamma}(\alpha),\bar{\boldsymbol{T}}_{1}(\alpha)\rangle & =\langle\dot{\gamma}_{1},\boldsymbol{T}_{1}\rangle|_{X(u_{1},t)}\\
 & =\langle\dot{\gamma}_{2},\bar{\boldsymbol{T}}_{1}(1)\rangle|_{X(u_{2},t)}.
\end{align*}
Therefore, we can calculate the mixed derivative 
\begin{align}
F_{s_{1}s_{2}} & =-\left(\frac{\partial}{\partial s_{2}}\frac{1}{d}\right)\langle\dot{\gamma},\boldsymbol{T}_{1}\rangle-\frac{1}{d}\nabla_{\tilde{\boldsymbol{T}}_{2}}\langle\dot{\gamma},\bar{\boldsymbol{T}}_{1}\rangle|_{X(u_{2},t)}\label{eq:F_s1s2 1}\\
 & =\frac{1}{d^{3}}\langle\dot{\gamma}_{1},\boldsymbol{T}_{1}\rangle|_{X(u_{1},t)}\langle\dot{\gamma}_{2},\boldsymbol{T}_{2}\rangle|_{X(u_{2},t)}-\frac{1}{d}\langle\boldsymbol{T}_{1},\boldsymbol{\tau}_{1}\rangle\langle\boldsymbol{T}_{2},\boldsymbol{\tau}_{2}\rangle\nonumber \\
 & \ -\frac{J_{2}'(1)}{d\cdot J_{2}(1)}\langle\boldsymbol{T}_{1},\boldsymbol{\nu}_{1}\rangle\langle\boldsymbol{T}_{2},\boldsymbol{\nu}_{2}\rangle-\frac{1}{d}\langle\dot{\gamma},\nabla_{\tilde{\boldsymbol{T}}_{2}}\bar{\boldsymbol{T}}_{1}\rangle|_{X(u_{2},t)}.\nonumber 
\end{align}
We compute 
\begin{align}
\frac{\partial}{\partial\alpha}\langle\dot{\gamma},\nabla_{\tilde{\boldsymbol{T}}_{2}}\bar{\boldsymbol{T}}_{1}\rangle & =\langle\dot{\gamma},\nabla_{\dot{\gamma}}\nabla_{\tilde{\boldsymbol{T}}_{2}}\bar{\boldsymbol{T}}_{1}\rangle\label{eq:=00005Cpartial_=00005Calpha}\\
 & =\langle{\rm Rm}(\dot{\gamma},\tilde{\boldsymbol{T}}_{2})\bar{\boldsymbol{T}}_{1},\dot{\gamma}\rangle+\langle\dot{\gamma},\nabla_{\tilde{\boldsymbol{T}}_{2}}\nabla_{\dot{\gamma}}\bar{\boldsymbol{T}}_{1}\rangle\nonumber \\
 & =-\langle{\rm Rm}(\dot{\gamma},\tilde{\boldsymbol{T}}_{2})\dot{\gamma},\bar{\boldsymbol{T}}_{1}\rangle\nonumber \\
 & =-\langle\nabla_{\dot{\gamma}}\nabla_{\dot{\gamma}}\tilde{\boldsymbol{T}}_{2},\bar{\boldsymbol{T}}_{1}\rangle.\nonumber 
\end{align}
Here, we have used the fact that $\nabla_{\dot{\gamma}}\nabla_{\dot{\gamma}}\tilde{\boldsymbol{T}}_{2}-{\rm Rm}(\dot{\gamma},\tilde{\boldsymbol{T}}_{2})\dot{\gamma}=0$,
where ${\rm Rm}(X,Y)$ is the curvature tensor $[\nabla_{X},\nabla_{Y}]-\nabla_{[X,Y]}$.
Note $\tilde{\boldsymbol{T}}_{2}|_{\alpha=0}=0$. Thus we have
\begin{align}
\frac{1}{d}\langle\dot{\gamma},\nabla_{\tilde{\boldsymbol{T}}_{2}}\bar{\boldsymbol{T}}_{1}\rangle|_{X(u_{2},t)} & =\frac{1}{d}\int_{0}^{1}\frac{\partial}{\partial\alpha}\langle\dot{\gamma},\nabla_{\tilde{\boldsymbol{T}}_{2}}\bar{\boldsymbol{T}}_{1}\rangle d\alpha=-\frac{1}{d}\int_{0}^{1}\langle\nabla_{\dot{\gamma}}\nabla_{\dot{\gamma}}\tilde{\boldsymbol{T}}_{2},\bar{\boldsymbol{T}}_{1}\rangle d\alpha\label{eq:F_s1s2 2 connecting}\\
 & =-\frac{1}{d}\int_{0}^{1}\frac{\partial}{\partial\alpha}\langle\nabla_{\dot{\gamma}}\tilde{\boldsymbol{T}}_{2},\bar{\boldsymbol{T}}_{1}\rangle d\alpha=-\frac{1}{d}\langle\nabla_{\dot{\gamma}}\tilde{\boldsymbol{T}}_{2},\bar{\boldsymbol{T}}_{1}\left.\rangle\right|_{\alpha=0}^{\alpha=1}\nonumber \\
 & =-\frac{1}{d}\langle\langle\boldsymbol{T}_{2},\boldsymbol{\tau}_{2}\rangle\boldsymbol{\tau}+\langle\boldsymbol{T}_{2},\boldsymbol{\nu}_{2}\rangle\frac{J_{2}'(\alpha)}{J_{2}(1)}\boldsymbol{\nu},\bar{\boldsymbol{T}}_{1}\left.\rangle\right|_{\alpha=0}^{\alpha=1}\nonumber \\
 & =-\langle\boldsymbol{T}_{2},\boldsymbol{\nu}_{2}\rangle\langle\boldsymbol{T}_{1},\boldsymbol{\nu}_{1}\rangle\frac{J_{2}'(1)-J_{2}'(0)}{d\cdot J_{2}(1)}.\nonumber 
\end{align}
Therefore, by (\ref{eq:F_s1s2 1}) and (\ref{eq:F_s1s2 2 connecting}),
we have 
\[
F_{s_{1}s_{2}}=-\frac{J_{2}'(0)}{d\cdot J_{2}(1)}\langle\boldsymbol{T}_{1},\boldsymbol{\nu}_{1}\rangle\langle\boldsymbol{T}_{2},\boldsymbol{\nu}_{2}\rangle.
\]
Similarly,
\[
F_{s_{2}s_{1}}=\frac{J_{1}'(1)}{d\cdot J_{1}(0)}\langle\boldsymbol{T}_{1},\boldsymbol{\nu}_{1}\rangle\langle\boldsymbol{T}_{2},\boldsymbol{\nu}_{2}\rangle.
\]
Note 
\begin{equation}
\frac{J_{1}'(1)}{J_{1}(0)}=-\frac{J_{2}'(0)}{J_{2}(1)}.\label{eq:J'_i(1)/J_i(0)}
\end{equation}
We collect the computations above in the following lemma:
\begin{lem}
Suppose that $X(s,t)$ is parametrized by length on the first variable.
Suppose that $p_{1}=X(s_{1},t)$ and $p_{2}=X(s_{2},t)$ are inside
a disk with radius smaller than the injectivity radius. Let $\gamma$
be the shortest geodesic connecting $p_{1},p_{2}$. Let $F(s_{1},s_{2},t)$
be the length of $\gamma$. \label{Lemma formulas}Then
\begin{align}
F_{t} & =\langle\boldsymbol{\tau}_{2},k_{2}\boldsymbol{n}_{2}\rangle-\langle\boldsymbol{\tau}_{1},k_{1}\boldsymbol{n}_{1}\rangle,\label{eq:derivatives}\\
F_{s_{1}} & =-\langle\boldsymbol{\tau}_{1},\boldsymbol{T}_{1}\rangle,F_{s_{2}}=\langle\boldsymbol{\tau}_{2},\boldsymbol{T}_{2}\rangle,\nonumber \\
F_{s_{1}s_{1}} & =-\frac{1}{d}\frac{J_{1}'(0)}{J_{1}(0)}\langle\boldsymbol{T}_{1},\boldsymbol{\nu}_{1}\rangle^{2}-\langle\boldsymbol{\tau}_{1},k_{1}\boldsymbol{n}_{1}\rangle,\nonumber \\
F_{s_{2}s_{2}} & =\frac{1}{d}\frac{J_{2}'(1)}{J_{2}(1)}\langle\boldsymbol{T}_{2},\boldsymbol{\nu}_{2}\rangle^{2}+\langle\boldsymbol{\tau}_{2},k_{2}\boldsymbol{n}_{2}\rangle,\nonumber \\
F_{s_{1}s_{2}} & =F_{s_{2}s_{1}}=\frac{1}{d}\frac{J_{1}'(1)}{J_{1}(0)}\langle\boldsymbol{T}_{1},\boldsymbol{\nu}_{1}\rangle\langle\boldsymbol{T}_{2},\boldsymbol{\nu}_{2}\rangle.\nonumber 
\end{align}
\end{lem}

In order to obtain estimates using Gaussian curvature bound, we will
need Rauch comparison principle \cite{carmo1992riemannian} on Jacobi
fields: 
\begin{lem}
\label{lem:Jacobian lemma}Suppose that the curvature of the Riemann
surface is bounded, i.e. 
\[
K_{M}=\sup_{M}|\mathcal{K}|<\infty.
\]
Let $J_{1}(\alpha)$ be defined by 
\[
J_{1}(1)=0,J_{1}'(1)=-d,J_{1}''(\alpha)+d^{2}\mathcal{K}(\gamma(\alpha))J_{1}(\alpha)=0.
\]
Then, by Rauch comparison principle \cite{carmo1992riemannian}, 
\[
|J_{1}(\alpha)|\leq\frac{1}{\sqrt{K_{M}}}\sinh(d\sqrt{K_{M}}(1-\alpha))=d(1-\alpha)+O(d^{3}),
\]
 for $0<\alpha<1$. Moreover,
\begin{align}
|J_{1}'(\alpha)-J_{1}'(1)| & =d^{2}|\int_{\alpha}^{1}\mathcal{K}(\gamma(t))J_{1}(t)dt|\leq K_{M}d^{3}(1-\alpha)+O(d^{5}).\label{eq:Jacobi derivative estimate}
\end{align}
So there exists a $\delta_{0}=\delta_{0}(K_{M})$ such that if $d<\delta_{0}$
then
\begin{align}
J_{1}'(\alpha) & \leq-d+2d^{3}K_{M},\label{eq:J' estimate}\\
J_{1}(0) & \geq d-2d^{3}K_{M}>\frac{1}{2}d.\nonumber 
\end{align}
\end{lem}

\section{Comparison function with exponential decay\label{sec:Comparison-function-with}}

In this section, we introduce our new comparison function for CSF
on surfaces. By Gage \cite{gage1990curve}, in a local neighborhood
away from the singularities of CSF, the distance of two distinct points
decreases at most exponentially. Inspired by \cite{gage1990curve,huisken1998distance},
we define a suitable comparison function:
\[
R(t)=\sup_{x,y}\frac{L(t)e^{-Kt}}{\pi d(x,y,t)}\sin\left(\frac{\pi l(x,y,t)}{L(t)}\right).
\]
If Type II singularity happens in finite time, $R(t)\to\infty$ locally,
as was explained in Section 2. Thus, if we want to prove that Type
II singularities do not occur, it is sufficient to prove that $R(t)$
is bounded. 
\begin{thm}
Let $d_{M}$ be the injectivity radius and $K_{M}=\sup_{M}|\mathcal{K}|$
be the Gaussian curvature bound of $M$. Suppose that the curve shortening
flow exists in the time period $[0,T)$. There exists a constant $K=K(d_{M},K_{M})$
such that $R(t)$ is bounded by a constant $C(R(0),d_{M},K_{M})$
. \label{thm:R(t) is boundeddd}
\end{thm}

\begin{proof}
Note that $L(t)$ is decreasing. Hence, $L(t)<L(0)$. The $\sin$
part is clearly bounded.

We now argue by contradiction. Suppose that 
\begin{equation}
\limsup_{t\to T}R(t)=\infty.\label{eq:R(t) contradiction 1 section 4}
\end{equation}
Pick $d_{0}=\frac{1}{2}\min\{d_{M},\delta_{0}\}$ where $\delta_{0}$
is defined in Lemma \ref{lem:Jacobian lemma}. If $d(x,y,t)>d_{0}$,
then $R(t)\leq\frac{L_{0}}{\pi d_{0}}$. Fix a large $N$ such that
$N>\max\{\frac{L(0)}{\pi d_{0}},R(0),2\}$ . Let 
\begin{equation}
Z_{N}(x,y,t)=Nd(x,y,t)-\frac{L(t)}{\pi}\sin\left(\frac{\pi l(x,y,t)}{L(t)}\right)e^{-Kt}.\label{eq:Z_N definition}
\end{equation}
Then, by (\ref{eq:R(t) contradiction 1 section 4}), there is a $\bar{t}$
such that for any $0<t<\bar{t}$ , $Z_{N}(x,y,t)>0$ for all $x,y$,
$Z_{N}(x,y,\bar{t})\geq0$ for all $x,y$ and $Z_{N}(\bar{x},\bar{y},\bar{t})=0$
for some $\bar{x},\bar{y}$. We can parametrize the curve $X(\cdot,\bar{t})$
by length with positive orientation such that 
\[
0\leq l=y-x\leq\frac{L}{2}.
\]
By the choice of $N$, we see $d(\bar{x},\bar{y},\bar{t})<\frac{d_{M}}{2}$
. Hence, $d(x,y,t)$ is smooth around $(\bar{x},\bar{y},\bar{t})$.
Since $\sin\frac{l\pi}{L}$ is smooth even when $l=\frac{1}{2}L$,
$Z_{N}$ is smooth in a neighborhood of $(\bar{x},\bar{y},\bar{t})$.
Let 
\[
p_{1}=X(\bar{x},\bar{t}),p_{2}=X(\bar{y},\bar{t}).
\]
We can then connect $X(\bar{x},\bar{t})$ and $X(\bar{y},\bar{t})$
by a shortest geodesic 
\[
\gamma:[0,1]\to M,\gamma(0)=X(\bar{x},\bar{t}),\gamma(1)=X(\bar{y},\bar{t}).
\]
 with constant speed $\|\dot{\gamma}\|=d(\bar{x},\bar{y},\bar{t})$.
By Lemma \ref{Lemma formulas}, compute the first variation:
\begin{align}
Z_{N}(x,y,t) & =Nd(x,y,t)-\frac{L(t)}{\pi}\sin\left(\frac{\pi l}{L}\right)e^{-Kt},\label{eq:First order derivatives}\\
\frac{\partial}{\partial x}Z_{N} & =-N\langle\boldsymbol{\tau}_{1},\boldsymbol{T}_{1}\rangle+\cos\left(\frac{\pi l}{L}\right)e^{-Kt},\nonumber \\
\frac{\partial}{\partial y}Z_{N} & =N\langle\boldsymbol{\tau}_{2},\boldsymbol{T}_{2}\rangle-\cos\left(\frac{\pi l}{L}\right)e^{-Kt}.\nonumber 
\end{align}
Second order derivatives are given by
\begin{align}
\frac{\partial^{2}}{\partial x^{2}}Z_{N} & =-\frac{N}{d}\frac{J_{1}'(0)}{J_{1}(0)}\langle\boldsymbol{T}_{1},\boldsymbol{\nu}_{1}\rangle^{2}-N\langle\boldsymbol{\tau}_{1},k_{1}\boldsymbol{n}_{1}\rangle+\frac{\pi e^{-Kt}}{L(t)}\sin\frac{\pi l(x,y,t)}{L(t)},\label{eq:Second order derivative}\\
\frac{\partial^{2}}{\partial y^{2}}Z_{N} & =\frac{N}{d}\frac{J_{2}'(1)}{J_{2}(1)}\langle\boldsymbol{T}_{2},\boldsymbol{\nu}_{2}\rangle^{2}+N\langle\boldsymbol{\tau}_{2},k_{2}\boldsymbol{n}_{2}\rangle+\frac{\pi e^{-Kt}}{L(t)}\sin\frac{\pi l(x,y,t)}{L(t)},\nonumber \\
\frac{\partial^{2}}{\partial x\partial y}Z_{N} & =\frac{N}{d}\frac{J_{1}'(1)}{J_{1}(0)}\langle\boldsymbol{T}_{1},\boldsymbol{\nu}_{1}\rangle\langle\boldsymbol{T}_{2},\boldsymbol{\nu}_{2}\rangle-\pi L(t)^{-1}\sin\frac{\pi l(x,y,t)}{L(t)}e^{-Kt}.\nonumber 
\end{align}
By (\ref{eq:L(t) decay}), we have
\begin{align*}
\frac{d}{dt}l(x,y,t) & =-\int_{x}^{y}k^{2}ds,\\
\frac{d}{dt}L(t) & =-\int_{S^{1}}k^{2}ds.
\end{align*}
Compute $t$ derivative:
\begin{align}
\frac{\partial}{\partial t}Z_{N} & =N\langle\boldsymbol{\tau}_{2},k_{2}\boldsymbol{n}_{2}\rangle-N\langle\boldsymbol{\tau}_{1},k_{1}\boldsymbol{n}_{1}\rangle+Ke^{-Kt}\frac{L(t)}{\pi}\sin\left(\frac{\pi l}{L}\right)\label{eq:Z_N t derivative}\\
 & +e^{-Kt}\left\{ \left(\frac{1}{\pi}\sin\frac{\pi l}{L}-\frac{l}{L}\cos\frac{\pi l}{L}\right)\int_{S^{1}}k^{2}du+\cos\left(\frac{\pi l}{L}\right)\int_{x}^{y}k^{2}ds\right\} .\nonumber 
\end{align}
At $(\bar{x},\bar{y},\bar{t})$, we have that 
\begin{align}
Nd & =\frac{L}{\pi}\sin\left(\frac{\pi l}{L}\right)e^{-K\bar{t}}.\label{eq:indentity for Nd}
\end{align}
Since $Z_{N}$ achieves minimum at $(\bar{x},\bar{y},\bar{t})$, by
Lemma \ref{Lemma formulas},
\begin{align}
\frac{\partial}{\partial x}Z_{N} & =\frac{\partial}{\partial y}Z_{N}=0,\label{eq:Z_N first derivative}\\
0<\langle\boldsymbol{\tau}_{i},\boldsymbol{T}_{i}\rangle & =\frac{1}{N}\cos\left(\frac{\pi l}{L}\right)e^{-K\bar{t}}\leq\frac{1}{N}.\label{eq:=00005Cthetaestimate}
\end{align}
We denote $\theta=\arccos\langle\boldsymbol{\tau}_{1},\boldsymbol{T}_{1}\rangle$.
Equation (\ref{eq:=00005Cthetaestimate}) implies either $\langle\boldsymbol{T}_{1},\boldsymbol{\nu}_{1}\rangle=\langle\boldsymbol{T}_{2},\boldsymbol{\nu}_{2}\rangle$
or $\langle\boldsymbol{T}_{1},\boldsymbol{\nu}_{1}\rangle=-\langle\boldsymbol{T}_{2},\boldsymbol{\nu}_{2}\rangle$.
We consider each case separately.

Case 1. $\langle\boldsymbol{T}_{1},\boldsymbol{\nu}_{1}\rangle=\langle\boldsymbol{T}_{2},\boldsymbol{\nu}_{2}\rangle$.
By Lemma \ref{lem:Jacobian lemma}, we have 
\[
|J_{i}'(0)-J_{i}'(1)|\leq C_{1}d^{3},|J_{1}(0)|\geq\frac{1}{2}d,
\]
where $C_{1}=C_{1}(K_{M})$. Let $K>4C_{1}$ . Define the operator
\[
\mathcal{L_{+}:=}\frac{\partial}{\partial t}-\frac{\partial^{2}}{\partial x^{2}}-\frac{\partial^{2}}{\partial y^{2}}+2\frac{\partial^{2}}{\partial x\partial y}.
\]
 By Lemma \ref{lem:Jacobian lemma}, (\ref{eq:Second order derivative})
, (\ref{eq:Z_N t derivative}) and (\ref{eq:indentity for Nd})we
have 
\begin{align*}
\mathcal{L}_{+}Z_{N} & |_{(\bar{x},\bar{y},\bar{t})}=\frac{N}{d}\left(\frac{J_{1}'(0)-J_{1}'(1)}{J_{1}(0)}-\frac{J_{2}'(1)-J_{2}'(0)}{J_{2}(1)}\right)\langle\boldsymbol{T}_{1},\boldsymbol{\nu}_{1}\rangle^{2}\\
 & +e^{-K\bar{t}}\left\{ \left(\frac{1}{\pi}\tan\frac{\pi l}{L}-\frac{l}{L}\right)\cos\frac{\pi l}{L}\int_{S^{1}}k^{2}ds\right.\\
 & \left.+\cos\left(\frac{\pi l}{L}\right)\int_{\bar{x}}^{\bar{y}}k^{2}ds+K\frac{L}{\pi}\sin\frac{\pi l}{L}\right\} \\
 & \geq-4C_{1}Nd+Ke^{-K\bar{t}}\frac{L}{\pi}\sin\frac{\pi l}{L}\\
 & =(K-4C_{1})e^{-K\bar{t}}\frac{L}{\pi}\sin\frac{\pi l}{L}>0.
\end{align*}
This contradicts with the fact that $Z_{N}(\cdot,\cdot,\bar{t})$
achieves local minimum at $(\bar{x},\bar{y})$.

Case 2. $\langle\boldsymbol{T}_{1},\boldsymbol{\nu}_{1}\rangle=-\langle\boldsymbol{T}_{2},\boldsymbol{\nu}_{2}\rangle.$
We need the following isoperimetric inequality:
\[
|\partial\Omega|^{2}\geq4\pi|\Omega|-K_{M}|\Omega|^{2},
\]
for a simply connected region $\Omega$ with piece-wise $C^{1}$ boundary
on $M$, see \cite{osserman1978isoperimetric}. Then, there exists
a $\delta_{1}=\delta_{1}(d_{M},K_{M})$ such that 
\[
|\partial\Omega|^{2}\geq2\pi|\Omega|,
\]
as long as $\Omega$ is inside a geodesic disk $D_{\delta_{1}}(p)$
such that
\[
|D_{\delta_{1}}(p)|<2\pi K_{M}^{-1}.
\]
Let $\delta_{M}=\min\{\delta_{1},\frac{1}{2(1+K_{M})}\}$. We will
discuss two subcases.

Subcase 2a. We assume $L(\bar{t})\geq\delta_{M}>0.$ Define the operator
\[
\mathcal{L_{-}:=}\frac{\partial}{\partial t}-\frac{\partial^{2}}{\partial x^{2}}-\frac{\partial^{2}}{\partial y^{2}}-2\frac{\partial^{2}}{\partial x\partial y}.
\]
Since $Z_{N}(\cdot,\cdot,\bar{t})$ attains local minimum at $(\bar{x},\bar{y})$,
we have $\mathcal{L}_{-}Z_{N}|_{(\bar{x},\bar{y},\bar{t})}\leq0$.
By Lemma \ref{lem:Jacobian lemma},
\begin{align}
\mathcal{L_{-}}Z_{N}|_{(\bar{x},\bar{y},\bar{t})}= & \frac{N}{d}\left(\frac{J_{1}'(0)-J_{1}'(1)}{J_{1}(0)}-\frac{J_{2}'(1)-J_{2}'(0)}{J_{2}(1)}\right)\langle\boldsymbol{T}_{1},\boldsymbol{\nu}_{1}\rangle^{2}\label{eq:R(t) bounded case 2-1}\\
 & +e^{-K\bar{t}}\left\{ \left(\frac{1}{\pi}\tan\frac{\pi l}{L}-\frac{l}{L}\right)\cos\frac{\pi l}{L}\int_{S^{1}}k^{2}ds\right.\nonumber \\
 & \left.+\cos\left(\frac{\pi l}{L}\right)\int_{x}^{y}k^{2}ds+K\frac{L}{\pi}\sin\frac{\pi l}{L}-\frac{4\pi}{L}\sin\frac{\pi l}{L}\right\} \nonumber \\
 & \geq-4C_{1}Nd+Ke^{-K\bar{t}}\frac{L}{\pi}\sin\frac{\pi l}{L}-e^{-K\bar{t}}\frac{4\pi}{L}\sin\frac{\pi l}{L}\nonumber \\
 & =(K-4C_{1}-\frac{4\pi^{2}}{L^{2}})e^{-K\bar{t}}\frac{L}{\pi}\sin\frac{\pi l}{L}\nonumber \\
 & \ge(K-4C_{1}-\frac{4\pi^{2}}{\delta_{M}^{2}})e^{-K\bar{t}}\frac{L}{\pi}\sin\frac{\pi l}{L}.\nonumber 
\end{align}
If 
\begin{equation}
K>4C_{1}+\frac{4\pi^{2}}{\delta_{M}^{2}},\label{eq:K condition case 2}
\end{equation}
then $\mathcal{L_{-}}Z_{N}|_{(\bar{x},\bar{y},\bar{t})}>0$ and we
have reached a contradiction.

Subcase 2b. We assume $L(\bar{t})<\delta_{M}$. Then, $\Gamma_{\bar{t}}$
is contained in a geodesic disk $D_{\delta_{M}}(p)$ centered at a
point $p$ . By Gauss-Bonnet formula,
\begin{equation}
\int_{S^{1}}k(u)du+\int_{\Omega_{\bar{t}}}\mathcal{K}(z)d\mu(z)=2\pi,\label{eq:Gauss bonnet in Omega_t}
\end{equation}
where $\Omega_{\bar{t}}$ is the region enclosed by the curve $X(\cdot,\bar{t})$.
We have the isoperimetric inequality 
\begin{equation}
|\Omega_{\bar{t}}|\leq L^{2}/2\pi.\label{eq:Area(Omega_t)}
\end{equation}
Thus by (\ref{eq:Gauss bonnet in Omega_t}) , (\ref{eq:Area(Omega_t)})
and the fact that $L<\delta_{M}$, we obtain
\[
\int_{S^{1}}|k|ds\geq2\pi-|\Omega_{\bar{t}}|\sup\mathcal{K}\geq2\pi-L^{2}\frac{K_{M}}{2\pi}>0.
\]
By Cauchy-Schwarz inequality, we have
\begin{align}
\int_{S^{1}}k^{2}ds & \geq\frac{1}{L}(\int_{S^{1}}|k|ds)^{2}\geq\frac{4\pi^{2}}{L}-2K_{M}L.\label{eq:total  S^1 curvature integral estimate}
\end{align}
Now, the shortest geodesic $\gamma$ together with the arc of $X(\cdot,\bar{t})$
from $\bar{x}$ to $\bar{y}$ enclose a region, denoted as $\Omega_{\bar{t}}^{+}$.
Note 
\[
\partial\Omega_{\bar{t}}^{+}=\{X(s,\bar{t}):s\in[\bar{x},\bar{y}]\}\cup\{\gamma(\alpha):\alpha\in[0,1]\}.
\]
Clearly, $|\partial\Omega_{\bar{t}}^{+}|\leq2l$. Two outer angles
at $p_{1}$ and $p_{2}$ on $\partial\Omega_{\bar{t}}^{+}$ are both
equal to $\pi-\theta$. By Gauss-Bonnet formula and isoperimetric
inequality (\ref{eq:Area(Omega_t)}), we have 
\begin{align}
\int_{\bar{x}}^{\bar{y}}kds & =2\pi-\int_{\Omega_{\bar{t}}^{+}}\mathcal{K}(z)d\mu(z)-2(\pi-\theta)\nonumber \\
 & =2\theta-\int_{\Omega_{\bar{t}}^{+}}\mathcal{K}(z)d\mu(z)\nonumber \\
 & \geq2\theta-\frac{2K_{M}}{\pi}l^{2}>0.\label{eq:half S^1 curvature estimate}
\end{align}
Note that the last inequality $2\theta-\frac{2K_{M}}{\pi}l^{2}>0$
holds since $\cos\theta<\frac{1}{N}$ and 
\[
\frac{K_{M}}{\pi}l^{2}<\frac{K_{M}}{(1+K_{M})^{2}4\pi}<\frac{\pi}{3}<\arccos\frac{1}{N}.
\]
Hence, by Cauchy-Schwarz inequality, 
\begin{equation}
\int_{\bar{x}}^{\bar{y}}k^{2}ds\geq\frac{1}{l}(2\theta-\frac{2K_{M}}{\pi}l^{2})^{2}\geq\frac{4\theta^{2}}{l}-\frac{8K_{M}\theta l}{\pi}.\label{eq:=00005Cintk^2 estimate}
\end{equation}
We have
\begin{align*}
\mathcal{L_{-}}Z_{N}|_{(\bar{x},\bar{y},\bar{t})} & =\frac{N}{d}\left(\frac{J_{1}'(0)-J_{1}'(1)}{J_{1}(0)}-\frac{J_{2}'(1)-J_{2}'(0)}{J_{2}(1)}\right)\langle\boldsymbol{T}_{1},\boldsymbol{\nu}_{1}\rangle^{2}\\
 & +e^{-K\bar{t}}\left[\frac{KL}{\pi}\sin\frac{\pi l}{L}-\frac{4\pi}{L}\sin\frac{\pi l}{L}+\cos\left(\frac{\pi l}{L}\right)\int_{x}^{y}k^{2}ds\right]\\
 & +e^{-K\bar{t}}\left(\frac{1}{\pi}\sin\frac{\pi l}{L}-\frac{l}{L}\cos\frac{\pi l}{L}\right)\int_{S^{1}}k^{2}ds.
\end{align*}
In the following, we use $C(K_{M})$ to represent some constants depending
only on $K_{M}$. By Lemma \ref{lem:Jacobian lemma}, (\ref{eq:total  S^1 curvature integral estimate})
and (\ref{eq:half S^1 curvature estimate}), we have 
\begin{align*}
0 & \geq\mathcal{L}_{-}Z_{N}|_{(\bar{x},\bar{y},\bar{t})}\geq-4C_{1}Nd\\
 & \ \ +e^{-K\bar{t}}\left[\frac{KL}{\pi}\sin\frac{\pi l}{L}-\frac{4\pi}{L}\sin\frac{\pi l}{L}+\cos\left(\frac{\pi l}{L}\right)\int_{x}^{y}k^{2}ds\right]\\
 & \ \ +e^{-K\bar{t}}\left(\frac{1}{\pi}\sin\frac{\pi l}{L}-\frac{l}{L}\cos\frac{\pi l}{L}\right)\left(\frac{4\pi^{2}}{L}-C(K_{M})L\right)\\
 & \geq-4C_{1}Nd+e^{-K\bar{t}}\left(\frac{KL}{\pi}-C(K_{M})L\right)\sin\frac{\pi l}{L}\\
 & \ +e^{-K\bar{t}}\left[\cos\left(\frac{\pi l}{L}\right)\int_{x}^{y}k^{2}ds-\frac{l}{L}\cos\frac{\pi l}{L}\left(\frac{4\pi^{2}}{L}-C(K_{M})L\right)\right].
\end{align*}
Then, by (\ref{eq:=00005Cintk^2 estimate}) and (\ref{eq:indentity for Nd}),
\begin{align*}
\mathcal{L}_{-}Z_{N}|_{(\bar{x},\bar{y},\bar{t})} & \geq-4C_{1}Nd+e^{-K\bar{t}}\left(K-C(K_{M})\right)\frac{L}{\pi}\sin\frac{\pi l}{L}\\
 & \ +e^{-K\bar{t}}\cos\frac{\pi l}{L}\left[4\theta^{2}-\frac{8\theta K_{M}}{\pi}l^{2}-\frac{4\pi^{2}l^{2}}{L^{2}}\right]\frac{1}{l}\\
 & =-4C_{1}Nd+e^{-K\bar{t}}\frac{L}{\pi}\sin\frac{\pi l}{L}\left(K-C(K_{M})-C(K_{M})\frac{4\pi\theta l}{L}\cot\frac{l\pi}{L}\right)\\
 & \ +e^{-K\bar{t}}\cos\frac{\pi l}{L}\left[4\theta^{2}-4\pi^{2}\frac{l^{2}}{L^{2}}\right]\frac{1}{l}\\
 & \geq e^{-K\bar{t}}\left[\frac{L}{\pi}\sin\frac{\pi l}{L}\left(K-C(K_{M})-4C_{1}\right)+\frac{1}{l}\cos\frac{\pi l}{L}\left(4\theta^{2}-4\pi^{2}\frac{l^{2}}{L^{2}}\right)\right].
\end{align*}
We can pick $K\geq C(K_{M})+4C_{1}$. In order to keep $\mathcal{L}_{-}Z_{N}|_{(\bar{x},\bar{y},\bar{t})}\leq0$,
it forces $4\theta^{2}-4\pi^{2}\frac{l^{2}}{L^{2}}\leq0$ and hence
$\theta\leq\frac{\pi l}{L}$. However, by (\ref{eq:=00005Cthetaestimate})
\[
\cos\theta=\frac{e^{-K\bar{t}}}{N}\cos\frac{\pi l}{L}<\cos\frac{\pi l}{L},
\]
which is impossible. 

By ruling out Case 1 and Case 2, we see that $Z_{N}(\cdot,\cdot,\bar{t})$
can not reach minimum $0$ at $(\bar{x},\bar{y})$, which implies
that $Z_{N}(t)>0$ for $t\in[0,T)$. Hence, $R(t)<N$. We have proved
the theorem.
\end{proof}
A non-existence result for Type II singularities now can be easily
derived. 
\begin{cor}
Embedded simple closed curves can not develop Type II singularities
in finite time.
\end{cor}

\begin{proof}
By performing Type II blow-ups \cite{altschuler1991singularities},
CSF near the singularity should converge to a grim reaper by Hamilton's
Harnack estimate \cite{hamilton1995harnack}. However, for a grim
reaper, the ratio of intrinsic and extrinsic distance blows up which
contradicts Theorem \ref{thm:R(t) is boundeddd}.
\end{proof}

\section{Analysis on Type I singularities \label{sec:Type-I}}

In this section, we focus on Type I singularities on Riemann surfaces.
The main tool we use here is a localized Huisken's backward heat kernel
monotonicity formula. Fix a point $x_{0}\in M$. Let $x\in\Gamma_{t}$
. We may assume $x=X(0,t)$. Let $\gamma:[0,1]\times(-\epsilon,\epsilon)\to M$
be a family of geodesic such that $\gamma(0,u)=x_{0}$, $\gamma(1,u)=X(u,t)$
and 
\[
\|\dot{\gamma}\|_{g}=\mathbf{d}_{g}(x_{0},X(t,u)),
\]
where $\mathbf{d}_{g}$ is the distance function on $(M,g)$. Here,
we use $\dot{\gamma}(\alpha,u)$ to denote $\frac{\partial}{\partial\alpha}\gamma(\alpha,u)$.
Let 
\[
\boldsymbol{T}:=\frac{\partial}{\partial s}=\frac{\partial X}{\partial u}/\|\frac{\partial X}{\partial u}\|_{g}
\]
By the first variation formula,
\[
\frac{\partial}{\partial s}\|\dot{\gamma}\|_{g}=\frac{1}{\|\dot{\gamma}\|_{g}}\langle\boldsymbol{T},\dot{\gamma}\rangle_{g}.
\]
Let $\boldsymbol{\nu}$ be a vector field orthogonal to $\dot{\gamma}$
and parallel along $\gamma$. Let $\mathcal{K}$ be the Gaussian curvature
of $(M,g)$. Let $J(\alpha)$ to be a function such that
\begin{equation}
\begin{cases}
J''(\alpha)+\mathbf{d}_{g}^{2}(x_{0},x)\mathcal{K}(\gamma(\alpha,0))J(\alpha)=0,\\
J(0)=0,J'(0)=\mathbf{d}_{g}(x_{0},x).
\end{cases}\label{eq:J(=00005Calpha) equation}
\end{equation}
Then we can extend $\mathcal{\boldsymbol{T}}$ to be a Jacobi field
by setting
\[
\mathcal{\boldsymbol{T}}(\gamma(\alpha,0))=\left\langle \mathcal{\boldsymbol{T}},\frac{\dot{\gamma}}{\|\dot{\gamma}\|}\right\rangle _{g}\alpha\frac{\dot{\gamma}(\alpha,0)}{\|\dot{\gamma}\|_{g}}+\left\langle \mathcal{\boldsymbol{T}},\boldsymbol{\nu}\right\rangle _{g}\frac{J(\alpha)}{J(1)}\boldsymbol{\nu}(\alpha).
\]
Then $\frac{J(\alpha)\boldsymbol{\nu}(\alpha)}{J(1)}$ will be a Jacobi
field along $\gamma(\alpha)$. Therefore, at $x$,

\begin{equation}
\nabla_{\mathcal{\boldsymbol{T}}}\dot{\gamma}=\nabla_{\dot{\gamma}}\mathcal{\boldsymbol{T}}=\mathcal{\boldsymbol{T}}+\frac{J'(1)-J(1)}{J(1)}\boldsymbol{\nu}.\label{eq:nabla =00005Csigma}
\end{equation}
Denote $\mathcal{J}_{x}:=\frac{J'(1)-J(1)}{J(1)}$. 
\begin{lem}
Let $K_{M}=\sup_{x\in M}|\mathcal{K}(x)|$ and $d=\mathbf{d}_{g}(x,x_{0})$.
Let $J(\alpha)$ be a solution of (\ref{eq:J(=00005Calpha) equation}).
By Lemma \ref{lem:Jacobian lemma}, if $d<\delta_{0}$, then 
\begin{equation}
|\mathcal{J}_{x}|\leq C(K_{M})d^{2}.\label{eq:=00005CmathcalJ_x}
\end{equation}
Let $g_{1}=c^{2}g$ and $\tilde{J}(\alpha)$ be a solution of
\[
\begin{cases}
\tilde{J}''(\alpha)+\mathbf{d}_{g_{1}}^{2}(x_{0},x)\mathcal{K}_{g_{1}}(\gamma(\alpha,0))\tilde{J}(\alpha)=0,\\
\tilde{J}(0)=0,\tilde{J}'(0)=\mathbf{d}_{g_{1}}(x_{0},x).
\end{cases}
\]
then $\tilde{J}(\alpha)=cJ(\alpha)$ and
\[
\tilde{\mathcal{J}_{x}}:=\frac{\tilde{J}'(1)-\tilde{J}(1)}{\tilde{J}(1)}=\mathcal{J}_{x}.
\]
\end{lem}

Since we are interested in the behavior of CSF near singularities,
we consider $t\in[t_{0},T)$ such that $T-1<t_{0}<T$. Define 
\[
\tau=-\frac{1}{2}\ln(T-t),t\in[t_{0},T).
\]
Then, $\tau\in[\tau_{0},\infty)$ for $\tau_{0}=-\frac{1}{2}\ln(T-t_{0})$.
Let $d_{M}$ be the injectivity radius of $M$ and let 
\[
r_{M}=\min\{\delta_{0},\frac{1}{8}d_{M}\}.
\]
Let $\eta_{0}(r)$ be a smooth cut-off function such that
\[
\eta_{0}(r)=\begin{cases}
1 & 0\leq r<1,\\
0 & r>2,
\end{cases}
\]
We may assume further that $0\geq\eta'_{0}>-2$. Now, We define 
\[
M_{x_{0}}(x):=\frac{1}{\sqrt{2(T-t)}}\exp\left(-\frac{\mathbf{d}_{g}(x_{0},x)^{2}}{4(T-t)}\right)\eta_{0}\left(\frac{\mathbf{d}_{g}(x_{0},x)^{2}}{4(T-t)\tau^{2}r_{M}^{2}}\right).
\]
We call $M_{x_{0}}(x)$ a localized Huisken's backward heat kernel
function. 

\begin{defn}
(Rescaled flow) \label{Def Rescaling}Let the rescaled pointed manifold
$M_{\tau}$ to be the pair $(M,x_{0},g_{\tau})$ with metric 
\[
g_{\tau}=(2(T-t))^{-1}g=\frac{e^{2\tau}}{2}g.
\]
We use $\langle,\rangle_{\tau}$ to denote the inner product under
metric $g_{\tau}$. Let $\Gamma_{\tau}=\Gamma_{t}$ such that $t=\frac{1}{2}e^{2\tau}$
but under the metric $g_{\tau}$. Likewise, we define $\sqrt{2}e^{-\tau}ds_{\tau}=ds$,
$k_{\tau}=k\sqrt{2}e^{-\tau}$, $\mathbf{d}_{\tau}(\cdot,\cdot)=\text{\textbf{d}}_{g_{\tau}}(\cdot,\cdot)$
to be the length form, geodesic curvature and distance function with
respect to the metric $g_{\tau}$.
\end{defn}

Under metric $g_{\tau}$, we have
\[
M_{x_{0}}ds=\exp\left(-\mathbf{d}_{\tau}(x_{0},x)^{2}/2\right)\eta_{0}\left(\frac{\mathbf{d}_{\tau}(x_{0},x)^{2}}{2\tau^{2}r_{M}^{2}}\right)ds_{\tau},
\]
Let 
\[
\rho(x):=\exp\left(-\mathbf{d}_{\tau}(x_{0},x)^{2}/2\right),
\]
and
\[
\eta(x):=\eta_{0}\left(\frac{\mathbf{d}_{\tau}(x_{0},x)^{2}}{2\tau^{2}r_{M}^{2}}\right)=\eta_{0}\left(\phi_{\tau,x_{0}}(x)\right),
\]
where $\phi_{\tau,x_{0}}(x)=\frac{\mathbf{d}_{\tau}(x_{0},x)^{2}}{2\tau^{2}r_{M}}$.
Finally, define 
\[
\mathcal{M}_{\tau}:=\int_{\Gamma_{\tau}}\rho\eta ds_{\tau}.
\]

\begin{thm}
Suppose that $\sup_{\tau,x}|k_{\tau}(x)|<C(k)<\infty$. Then, there
exist constants $C_{1},C_{2}$ depending on $K_{M},r_{M},L(0)$ such
that 
\begin{equation}
\frac{d}{d\tau}\left(p(\tau)\mathcal{M}_{\tau}\right)\leq C_{2}p(\tau)\exp(2\tau-\tau^{2}r_{M}^{2})\tau^{-1},\label{eq:localmonontonicityformula1}
\end{equation}
where $p(\tau)=\exp\left(-C_{1}\int_{0}^{\tau}e^{-2y}y^{2}dy\right)$.
\end{thm}

\begin{proof}
For $x\in\Gamma_{t}$, we may assume $x=X(0,t)$. Let $\gamma:[0,1]\times(-\epsilon,\epsilon)\to M$
be a family of geodesics such that $\gamma(0,u)=x_{0}$, $\gamma(1,u)=X(u,t)$
and $\|\dot{\gamma}\|_{g}=\mathbf{d}_{g}(x_{0},X(t,u))$. Since $\eta(x)\not=0$
only if $\mathbf{d}_{\tau}(x_{0},x)^{2}\leq4\tau^{2}r_{M}^{2}$, we
may assume 
\[
\mathbf{d}_{g}^{2}(x_{0},x)\leq8e^{-2\tau}\tau^{2}r_{M}^{2}\leq d_{M}^{2}.
\]
Hence, we can use the first variation formula and obtain
\[
\frac{\partial}{\partial\tau}\|\dot{\gamma}\|_{\tau}^{2}=\frac{\partial}{\partial\tau}\left(\frac{1}{2}e^{-2\tau}\|\dot{\gamma}\|_{g}^{2}\right)=\left(\langle\dot{\gamma},k_{\tau}\boldsymbol{n}\rangle_{\tau}+\|\dot{\gamma}\|_{\tau}^{2}\right).
\]
By (\ref{eq:nabla =00005Csigma}),
\begin{align}
\frac{\partial}{\partial s_{\tau}}\rho & =-\rho\langle\dot{\gamma},\nabla_{\mathcal{\boldsymbol{T}}}\dot{\gamma}\rangle_{\tau}=-\rho\langle\dot{\gamma},\mathcal{\boldsymbol{T}}\rangle_{\tau}.\label{=00005Cpartials_tau =00005Crho}
\end{align}
Combining (\ref{eq:nabla =00005Csigma}) and (\ref{=00005Cpartials_tau =00005Crho}),
we have 
\begin{align}
\int_{\Gamma_{\tau}}\rho\eta ds & =\int_{\Gamma_{\tau}}\rho\eta\langle\mathcal{\boldsymbol{\boldsymbol{T}}},\mathcal{\boldsymbol{\boldsymbol{T}}}\rangle_{\tau}ds_{\tau}\label{eq:=00005Cint=00005Crho=00005Ceta}\\
 & =\int_{\Gamma_{\tau}}\rho\eta\langle\mathcal{\boldsymbol{T}},\nabla_{\mathcal{\boldsymbol{T}}}\dot{\gamma}\rangle_{\tau}ds_{\tau}-\int_{\Gamma_{\tau}}\rho\eta\langle\mathcal{\boldsymbol{T}},\mathcal{J}_{x}\boldsymbol{\nu}\rangle_{\tau}ds_{\tau}\nonumber \\
 & =\int_{\Gamma_{\tau}}\frac{\partial}{\partial s_{\tau}}\left(\rho\eta\langle\mathcal{\boldsymbol{T}},\dot{\gamma}\rangle_{\tau}\right)ds_{\tau}-\int_{\Gamma_{\tau}}\rho\eta\langle\nabla_{\mathcal{\boldsymbol{T}}}\mathcal{\mathcal{\boldsymbol{T}}},\dot{\gamma}\rangle_{\tau}ds_{\tau}-\int_{\Gamma_{\tau}}\langle\mathcal{\boldsymbol{T}},\dot{\gamma}\rangle_{\tau}\eta\frac{\partial}{\partial s_{\tau}}\rho ds_{\tau}\nonumber \\
 & \ -\int_{\Gamma_{\tau}}\langle\mathcal{\boldsymbol{T}},\dot{\gamma}\rangle_{\tau}\rho\frac{\partial}{\partial s_{\tau}}\eta ds_{\tau}-\int_{\Gamma_{\tau}}\rho\eta\langle\mathcal{\boldsymbol{T}},\mathcal{J}_{x}\boldsymbol{\nu}\rangle_{\tau}ds_{\tau}\nonumber \\
 & =-\int_{\Gamma_{\tau}}\rho\eta\langle k\boldsymbol{n},\dot{\gamma}\rangle_{\tau}ds_{\tau}+\int_{\Gamma_{\tau}}\rho\eta\langle\mathcal{\boldsymbol{T}},\dot{\gamma}\rangle^{2}ds_{\tau}-\int_{\Gamma_{\tau}}\langle\mathcal{\boldsymbol{T}},\dot{\gamma}\rangle_{\tau}\rho\frac{\partial}{\partial s_{\tau}}\eta ds_{\tau}\nonumber \\
 & \ -\int_{\Gamma_{\tau}}\rho\eta\langle\mathcal{\boldsymbol{T}},\mathcal{J}_{x}\boldsymbol{\nu}\rangle_{\tau}ds_{\tau}.\nonumber 
\end{align}
Since $\sqrt{2}e^{-\tau}ds_{\tau}=ds,$ we have 
\begin{align}
\frac{\partial}{\partial\tau}ds_{\tau} & =(1-k_{\tau}^{2})ds_{\tau}.\label{eq:partial tau ds_=00005Ctau}
\end{align}
Therefore, by (\ref{eq:partial tau ds_=00005Ctau}) and (\ref{eq:=00005Cint=00005Crho=00005Ceta}),
\begin{align}
\frac{d}{d\tau}\mathcal{M}_{\tau} & =\frac{\partial}{\partial\tau}\int_{\Gamma_{\tau}}\rho\eta ds_{\tau}\label{M_=00005Ctau1}\\
 & =\int_{\Gamma_{\tau}}\rho\eta\left(-k_{\tau}^{2}+1-\|\dot{\gamma}\|_{\tau}^{2}-\langle\dot{\gamma},k\boldsymbol{n}\rangle_{\tau}\right)ds_{\tau}+\int_{\Gamma_{\tau}}\rho\frac{\partial}{\partial\tau}\eta ds_{\tau}\nonumber \\
 & =\int_{\Gamma_{\tau}}\rho\eta\left(-k^{2}-2k\langle\boldsymbol{n},\dot{\gamma}\rangle+\langle\mathcal{\boldsymbol{T}},\dot{\gamma}\rangle^{2}-\|\dot{\gamma}\|^{2}\right)ds_{\tau}-\int_{\Gamma_{\tau}}\langle\mathcal{\boldsymbol{T}},\dot{\gamma}\rangle\rho\frac{\partial}{\partial s}\eta ds_{\tau}\nonumber \\
 & \ -\int_{\Gamma_{\tau}}\rho\eta\langle\mathcal{\boldsymbol{T}},\mathcal{J}_{x}\boldsymbol{\nu}\rangle ds_{\tau}+\int_{\Gamma_{\tau}}\rho\frac{\partial}{\partial\tau}\eta ds_{\tau}\nonumber \\
 & =-\int_{\Gamma_{\tau}}\rho\eta|k+\langle\dot{\gamma},\boldsymbol{n}\rangle_{\tau}|^{2}ds_{\tau}-\int_{\Gamma_{\tau}}\rho\eta\langle\mathcal{\boldsymbol{T}},\mathcal{J}_{x}\boldsymbol{\nu}\rangle ds_{\tau}+\int_{\Gamma_{\tau}}A_{\eta}ds_{\tau},\nonumber 
\end{align}
where $A_{\eta}=\rho\frac{\partial}{\partial\tau}\eta-\langle\mathcal{\boldsymbol{T}},\dot{\gamma}\rangle_{\tau}\rho\frac{\partial}{\partial s}\eta$.
Note
\[
\eta(x)=\eta_{0}\left(\frac{\|\dot{\gamma}\|_{g_{\tau}}^{2}}{2\tau^{2}r_{M}^{2}}\right)=\eta_{0}\left(\frac{e^{2\tau}\|\dot{\gamma}\|_{g}^{2}}{2\tau^{2}r_{M}^{2}}\right).
\]
We compute the derivatives of $\eta$:
\begin{align*}
\frac{\partial}{\partial\tau}\eta(x) & =\eta'_{0}\left(\phi_{\tau,x_{0}}(x)\right)\frac{\tau^{-2}}{r_{M}^{2}}\left(\langle\dot{\gamma},k_{\tau}\boldsymbol{n}\rangle_{g_{\tau}}+(1-\tau^{-1})\|\dot{\gamma}\|_{g_{\tau}}^{2}\right),\\
\frac{\partial}{\partial s}\eta(x) & =\eta'_{0}\left(\phi_{\tau,x_{0}}(x)\right)\frac{\langle\dot{\gamma},\nabla_{\mathcal{\boldsymbol{T}}}\dot{\gamma}\rangle_{g_{\tau}}}{\tau^{2}r_{M}^{2}}\\
 & =\eta_{0}'\left(\phi_{\tau,x_{0}}(x)\right)\frac{\tau^{-2}}{r_{M}^{2}}\langle\dot{\gamma},\mathcal{\boldsymbol{T}}+\mathcal{J}_{x}\boldsymbol{\nu}\rangle_{\tau}\langle\mathcal{\boldsymbol{T}},\dot{\gamma}\rangle_{\tau}.
\end{align*}
Thus,
\begin{align*}
A_{\eta} & =\rho\eta_{0}'\left(\phi_{\tau,x_{0}}(x)\right)\frac{\tau^{-2}}{r_{M}^{2}}\left[\langle\dot{\gamma},k_{\tau}\boldsymbol{n}\rangle_{\tau}+(1-\tau^{-1})\|\dot{\gamma}\|_{\tau}^{2}-\langle\dot{\gamma},\mathcal{\boldsymbol{T}}+\mathcal{J}_{x}\boldsymbol{\nu}\rangle_{\tau}\langle\mathcal{\boldsymbol{T}},\dot{\gamma}\rangle_{\tau}\right]\\
 & =\rho\eta_{0}'\left(\phi_{\tau,x_{0}}(x)\right)\frac{\tau^{-2}}{r_{M}^{2}}\left[\langle\dot{\gamma},k_{\tau}\boldsymbol{n}\rangle_{\tau}+\langle\dot{\gamma},\boldsymbol{\nu}\rangle^{2}-\tau^{-1}\|\dot{\gamma}\|_{\tau}^{2}-\langle\mathcal{J}_{x}\boldsymbol{\nu},\dot{\gamma}\rangle_{\tau}\langle\mathcal{\boldsymbol{T}},\dot{\gamma}\rangle_{\tau}\right].
\end{align*}
Now, (\ref{M_=00005Ctau1}) becomes
\begin{align}
\frac{d}{d\tau}\mathcal{M}_{\tau} & =-\int_{\Gamma_{\tau}}\rho\eta|k+\langle\dot{\gamma},\boldsymbol{n}\rangle_{\tau}|^{2}ds_{\tau}-\int_{\Gamma_{\tau}}\rho\eta\langle\mathcal{\boldsymbol{T}},\mathcal{J}_{x}\boldsymbol{\nu}\rangle ds_{\tau}+\int_{\Gamma_{\tau}}A_{\eta}ds_{\tau}.\label{eq:M_=00005Ctau2}\\
 & =-\int_{\Gamma_{\tau}}\rho\eta|k+\langle\dot{\gamma},\boldsymbol{n}\rangle_{\tau}|^{2}ds_{\tau}+\int_{\Gamma_{\tau}}\rho\eta_{0}'\left(\phi_{\tau,x_{0}}(x)\right)\frac{\tau^{-2}}{r_{M}^{2}}\langle\dot{\gamma},\boldsymbol{\nu}\rangle^{2}ds_{\tau}\nonumber \\
 & -\int_{\Gamma_{\tau}}\rho\eta\langle\mathcal{\boldsymbol{T}},\mathcal{J}_{x}\boldsymbol{\nu}\rangle_{\tau}ds_{\tau}+\int_{\Gamma_{\tau}}\rho\eta_{0}'\left(\phi_{\tau,x_{0}}(x)\right)\frac{\tau^{-2}}{r_{M}^{2}}\left(\langle\dot{\gamma},k_{\tau}\boldsymbol{n}\rangle_{\tau}-\langle\mathcal{J}_{x}\boldsymbol{\nu},\dot{\gamma}\rangle_{\tau}\langle\mathcal{\boldsymbol{T}},\dot{\gamma}\rangle_{\tau}\right)ds_{\tau}\nonumber \\
 & -\int_{\Gamma_{\tau}}\rho\eta_{0}'\left(\phi_{\tau,x_{0}}(x)\right)\frac{\tau^{-2}}{r_{M}^{2}}\tau^{-1}\|\dot{\gamma}\|_{\tau}^{2}ds_{\tau}\nonumber \\
 & =M_{1}+M_{2}+M_{3}.\nonumber 
\end{align}
Here 
\[
M_{1}=-\int_{\Gamma_{\tau}}\rho\eta|k+\langle\dot{\gamma},\boldsymbol{n}\rangle_{\tau}|^{2}ds_{\tau}+\int_{\Gamma_{\tau}}\rho\eta_{0}'\left(\phi_{\tau,x_{0}}(x)\right)\frac{\tau^{-2}}{r_{M}^{2}}\langle\dot{\gamma},\boldsymbol{\nu}\rangle^{2}ds_{\tau}\leq0,
\]
which is a good term. The remaining terms in (\ref{eq:M_=00005Ctau2})
are
\begin{align}
M_{2} & :=-\int_{\Gamma_{\tau}}\rho\eta\langle\mathcal{\boldsymbol{T}},\mathcal{J}_{x}\boldsymbol{\nu}\rangle_{\tau}ds_{\tau}\label{eq:M_2 def}\\
 & +\int_{\Gamma_{\tau}}\rho\eta_{0}'\left(\phi_{\tau,x_{0}}(x)\right)\frac{\tau^{-2}}{r_{M}^{2}}\left(\langle\dot{\gamma},k_{\tau}\boldsymbol{n}\rangle_{\tau}-\langle\mathcal{J}_{x}\boldsymbol{\nu},\dot{\gamma}\rangle_{\tau}\langle\mathcal{\boldsymbol{T}},\dot{\gamma}\rangle_{\tau}\right)ds_{\tau},\nonumber \\
M_{3} & :=-\int_{\Gamma_{\tau}}\rho\eta_{0}'\left(\phi_{\tau,x_{0}}(x)\right)\frac{\|\dot{\gamma}\|_{\tau}^{2}}{\tau^{2}r_{M}^{2}}\tau^{-1}ds_{\tau}.
\end{align}
Since $\eta_{0}(r)=0$ for $r\geq2$ and $|\eta_{0}'|<2$ , we see
that 
\begin{equation}
\left|\eta_{0}'\left(\phi_{\tau,x_{0}}(x)\right)\frac{\|\dot{\gamma}\|_{\tau}^{2}}{\tau^{2}r_{M}^{2}}\right|\leq8,\eta_{0}\left(\phi_{\tau,x_{0}}(x)\right)\frac{\|\dot{\gamma}\|_{\tau}^{2}}{\tau^{2}r_{M}^{2}}\leq4.\label{d_=00005Ctauestimate}
\end{equation}
Let $d_{\tau}=\mathbf{d}_{\tau}(x_{0},x)$. Then by (\ref{eq:=00005CmathcalJ_x})
, we have
\begin{align*}
|\mathcal{J}_{x}| & \leq d^{2}C(K_{M})\leq C(K_{M})d_{\tau}^{2}e^{-2\tau}.
\end{align*}
Thus by (\ref{d_=00005Ctauestimate}), we may assume
\begin{equation}
|\mathcal{J}_{x}|\leq\tau^{2}e^{-2\tau}r_{M}^{2}C(K_{M}).\label{eq:J_x}
\end{equation}
Define $G_{\tau}(x):=\left|\rho\eta_{0}'\left(\phi_{\tau,x_{0}}(x)\right)\right|.$
Since $\eta'_{0}=0$ if $d_{\tau}<2\tau^{2}r_{M}^{2}$, we have $G_{\tau}(x)\leq2\exp(-\tau^{2}r_{M}^{2})$
and 
\begin{equation}
\int_{\Gamma_{\tau}}G_{\tau}ds_{\tau}\leq L(0)e^{2\tau-\tau^{2}r_{M}^{2}}.\label{eq:B_=00005Ctau}
\end{equation}
Hence, by (\ref{eq:J_x}), (\ref{eq:B_=00005Ctau}) and the fact that
$k_{\tau}<C(k)$, we obtain that
\begin{align}
|M_{2}| & \leq C(K_{M},r_{M})e^{-2\tau}\tau^{2}\mathcal{M}_{\tau}+C(r_{M},C(k),K_{M})\tau^{-1}\int_{\Gamma_{\tau}}G_{\tau}ds_{\tau}\label{eq:M_2estimate}\\
 & \leq C(K_{M},r_{M})e^{-2\tau}\tau^{2}\mathcal{M}_{\tau}+C(r_{M},C(k),K_{M},L(0))e^{2\tau-\tau^{2}r_{M}^{2}}\tau^{-1}.\nonumber 
\end{align}
Similarly, for $M_{3}$, we have 
\begin{align}
|M_{3}| & \leq C\tau^{-1}\int_{\Gamma_{\tau}}G_{\tau}ds_{\tau}\label{eq:M_3estimate}\\
 & \leq C(L(0))\tau^{-1}\exp(2\tau-\tau^{2}r_{M}^{2}).\nonumber 
\end{align}
Since $M_{1}\leq0$, by (\ref{eq:M_=00005Ctau2}),(\ref{eq:M_2estimate})
and (\ref{eq:M_3estimate}), we have
\[
\frac{d}{d\tau}\mathcal{M}_{\tau}\leq C_{1}e^{-2\tau}\tau^{2}\mathcal{M}_{\tau}+C_{2}\exp(2\tau-\tau^{2}r_{M}^{2})\tau^{-1}.
\]
Let $p(\tau):=\exp\left(-C_{1}\int_{0}^{\tau}e^{-2y}y^{2}dy\right).$
Then
\[
\frac{d}{d\tau}\left(p(\tau)\mathcal{M}_{\tau}\right)\leq C_{2}p(\tau)\exp(2\tau-\tau^{2}r_{M}^{2})\tau^{-1}.
\]
\end{proof}
Note 
\begin{align*}
\lim_{\tau\to\infty}p(\tau) & =\exp(-C_{1}\int_{0}^{\infty}e^{-2y}y^{2}dy)\\
 & =\exp\left(-C_{1}\frac{\Gamma(3)}{2^{3}}\right)<\infty.
\end{align*}
Integrate (\ref{eq:localmonontonicityformula1}) to get

\begin{equation}
\int_{\tau_{0}}^{\tau_{1}}\frac{d}{d\tau}\left(p(\tau)\mathcal{M}_{\tau}\right)d\tau\leq C_{2}\int_{\tau_{0}}^{\tau_{1}}p(\tau)\exp(2\tau-\tau^{2}r_{M}^{2})\tau^{-1}d\tau.\label{eq:integral estimate for p(t)M_t}
\end{equation}
The RHS of (\ref{eq:integral estimate for p(t)M_t}) is clearly finite
as $\tau_{1}\to\infty$. Thus, $\mathcal{M}_{\tau}$ is bounded. Note
$M_{2},M_{3}=o(e^{-\tau})$ as $\tau\to\infty$. Therefore, if $\tau_{1},\tau_{2}\to\infty$,
\[
\left.p(\tau)\mathcal{M}_{\tau}\right|_{\tau_{1}}^{\tau_{2}}=\int_{\tau_{1}}^{\tau_{2}}p(\tau)M_{1}d\tau+\int_{\tau_{1}}^{\tau_{2}}p(\tau)(M_{2}+M_{3})d\tau\to0,
\]
which implies $\lim_{\tau_{1},\tau_{2}\to\infty}\int_{\tau_{1}}^{\tau_{2}}p(\tau)M_{1}d\tau\to0$.
Hence, 
\begin{equation}
\int_{\tau_{1}}^{\tau_{2}}p(\tau)\left(\int_{\Gamma_{\tau}}-\rho\eta|k+\langle\dot{\gamma},\boldsymbol{n}\rangle_{\tau}|^{2}ds_{\tau}+\int_{\Gamma_{\tau}}\rho\eta_{0}'\left(\text{\ensuremath{\phi_{\tau,x_{0}}(x)}}\right)\frac{\langle\dot{\gamma},\boldsymbol{\nu}\rangle^{2}}{r_{M}^{2}\tau^{2}}ds_{\tau}\right)\to0.\label{monontonicity formula 2}
\end{equation}
as $\tau_{1},\tau_{2}\to\infty$.
\begin{thm}
\label{thm:If-Type-I}If a Type I singularity happens, CSF converges
to a round point.
\end{thm}

\begin{proof}
The proof is quite standard, see \cite{angenent1991formation,altschuler1991singularities}.
The difference here is that instead of rescaling the curve, we need
to rescale the metric near a blow-up point. Suppose that there exists
a sequence $\{t_{i}\}$ such that $p_{i}\in\Gamma_{t_{i}}$ and the
geodesic curvature $|k(p_{i})|=\sup_{x\in\Gamma_{t_{i}}}|k(x)|\to\infty$
as $t_{i}\to T$. We may assume that up to a subsequence $p_{i}\to x_{0}$
for some $x_{0}\in M$. By the definition of Type I singularity, 
\begin{equation}
\sup_{\Gamma_{t}}k^{2}(T-t)<C<\infty.\label{Type I condition}
\end{equation}
 Let $\tau=-\frac{1}{2}\ln(T-t)$. We blow up at $x_{0}$ by rescaling
the metric centered at $x_{0}$ using Definition \ref{Def Rescaling}.
Clearly, the pointed metric space $(M,x_{0},g_{\tau})$ converges
to the tangent plane $(\mathbb{R}^{2},0,g_{\mathbb{R}^{2}})$. Since
$\Gamma_{\tau}$ has bounded curvature by (\ref{Type I condition})
and by the estimates from parabolic equations, $\Gamma_{\tau}$ converges
to a limit curve $\Gamma^{\infty}$ on the limit metric space $\mathbb{R}^{2}$
as $\tau\to\infty$ \cite{angenent1991formation,altschuler1991singularities}.
Note by (\ref{Type I condition}), we can estimate the distance between
$p_{i}$ and $x_{0}$:
\begin{align}
d_{\tau_{i}}(p_{i},x_{0}) & =\frac{1}{\sqrt{2}}e^{\tau_{i}}\mathbf{d}_{g}(p_{i},x_{0})\leq\frac{1}{\sqrt{2}}e^{\tau_{i}}\int_{t_{i}}^{T}|k|dt\label{eq:distance (p_i,x_0)}\\
 & \leq Ce^{\tau_{i}}\sqrt{T-t_{i}}=C'.\nonumber 
\end{align}
Hence, $\Gamma^{\infty}$ is not empty by (\ref{eq:distance (p_i,x_0)}).
By (\ref{monontonicity formula 2}), the curve $\Gamma^{\infty}$
satisfies
\begin{equation}
k^{\infty}+\langle\dot{\gamma}^{\infty},\boldsymbol{n}\rangle=0,\label{eq:self similar solution}
\end{equation}
where $k^{\infty}$ is the curvature of $\Gamma^{\infty}$ and $\dot{\gamma}^{\infty}$
now represents the position vector of $\Gamma^{\infty}$ on the tangent
plane $\mathbb{R}^{2}$. (\ref{eq:self similar solution}) represents
a self-similar solution of CSF on $\mathbb{R}^{2}$. Since $\Gamma_{\tau}$
are embedded closed curves, $\Gamma^{\infty}$ has only two possibilities:
a round circle or a straight line by Abresch-Langer classification
of self-similar solutions \cite{abresch1986normalized}. Maximum principle
shows that $\sup_{\Gamma_{\tau}}|k_{\tau}|\geq C>0$, see \cite{chou2001curve}
chapter 5 . Thus, $\Gamma^{\infty}$ can not be a straight line and
hence is a round circle.
\end{proof}
Now, we can prove the Gage-Hamilton-Grayson Theorem on surfaces. 
\begin{proof}
[Proof of Theorem \ref{thm:G-H-G thm}]For any CSF on a surface, by
Theorem \ref{thm:R(t) is boundeddd}, it can not develop Type II singularities.
Therefore, the first time singularity, if exists, is of Type I, which
by Theorem \ref{thm:If-Type-I}, implies that the CSF converges to
a round point. Hence, we have proved Gage-Hamilton-Grayson Theorem
on surfaces.
\end{proof}

\section{Conic Riemann surfaces}

\label{sec:conic surfaces}In this section, we focus on conic Riemann
surfaces. Let $M$ be a closed Riemann surface. Let $\{p_{i}\}_{i=1}^{k}\subset M$,
\[
D=\sum_{i=1}^{k}\beta_{i}p_{i},
\]
be a divisor on $M$ with $\beta_{i}>-1$. 
\begin{defn}
We call $(M,g,D)$ a \textbf{conic Riemann surface} with divisor $D$
if each $\beta_{i}\in(-1,0)$, $g$ is smooth in $M\backslash\{p_{i}\}$,
and near each $p_{i}$ , there exists a local holomorphic coordinate
$z$ such that the metric 
\begin{equation}
g=|z-p_{i}|^{2\beta_{i}}e^{2h(z)}|dz|^{2},\label{eq:conic metric def}
\end{equation}
for some continuous function $h$ in certain function spaces. If\textbf{
}for some point $p_{i}$, \textbf{$g$ }still has local expression
(\ref{eq:conic metric def}) but $\beta_{i}>0$, we call\textbf{ }$(M,g,D)$
a generalized Riemann surface. 
\end{defn}

We will restrict to conic Riemann surfaces from now on. The function
spaces for $h(z)$ in (\ref{eq:conic metric def}) is $C^{\infty}(M-\{p_{i}\})\cap C^{2,\alpha}(M,D)$,
where $C^{2,\alpha}(M,D)$ is the weighted H$\ddot{o}$lder space
defined by H. Yin \cite{yin2010ricci}:
\begin{defn}
(Weighted H$\ddot{o}$lder space).\label{def:(Weighted-Hlder-space).}
Let $p$ be a conic point of $M$. Let $U\subset M$ be a disk neighborhood
of $p$ such that the divisor $D|_{U}=\beta p$. Let $z=re^{\sqrt{-1}\theta}$
be a local holomorphic coordinate centered at $p$. Define the weighted
H$\ddot{o}$lder norm for $f$ to be 
\[
\|f\|_{C^{\ell,\alpha}(U,\beta p)}:=\sup_{m\in\mathbf{N}}\left\Vert F_{m}(s,\theta)\right\Vert _{C^{\ell,\alpha}\left(\left(2^{-1}\leq s\leq2\right)\times S^{1}\right)},
\]
where $F_{m}(s,\theta)\equiv f\left((1+\beta)^{\frac{1}{1+\beta}}\left(2^{-m}s\right)^{\frac{1}{1+\beta}}e^{\sqrt{-1}\theta}\right)$.
It is proved in \cite{yin2010ricci} that the definition of $C^{l,\alpha}(U,\beta p)$
does not depend on the choice of a local coordinate.

On $M$, take $(U_{i},p_{i})$ to be a disk neighborhood of each conic
point $p_{i}$. Cover $M-\cup_{i=1}^{k}U_{i}$ by finite many open
sets $V_{k}$. We define the weighted H$\ddot{o}$lder norm:
\[
\|f\|_{C^{l,\alpha}(M,D)}:=\sum_{k}\|f\|_{C^{l,\alpha}(V_{k})}+\sum_{i}\|f\|_{C^{l,\alpha}(U_{i},\beta_{i}p_{i})}.
\]
Then, $f$ belongs to $C^{l,\alpha}(M,D)$ if it has finite weighted
H$\ddot{o}$lder norm. 
\end{defn}

For each $p_{i}$, the number $\theta_{i}=2\pi(1+\beta_{i})$ represents
the\textbf{ cone angle }at $p_{i}$. The curvature of $g$ near $p_{i}$
is given by
\begin{align*}
\mathcal{K}_{g}(z) & =|z-p_{i}|^{-2\beta_{i}}e^{-2h(z)}\left[-2\partial_{z}\partial_{\bar{z}}\log\left(|z-p_{i}|^{2\beta_{i}}e^{2h(z)}\right)\right]\\
 & =e^{-2h(z)}|z-p_{i}|^{-2\beta_{i}}\left(-4\partial_{z}\partial_{\bar{z}}h(z)\right).
\end{align*}
The assumption $h\in C^{2,\alpha}(M,D)$ indicates that $|z-p_{i}|^{-2\beta_{i}}\partial_{z}\partial_{\bar{z}}h(z)$
is bounded. Hence, the curvature $\mathcal{K}_{g}$ is bounded. 

The following proposition will be assumed and the detailed discussion
can be found in the appendix. We mention that a geodesic polar coordinate
near each conic point which is obtained by Troyanov \cite{troyanov1990coordonnees}.
\begin{prop}
Let $(M,g,D)$ be a conic Riemann surface with $D=\sum_{i=1}^{k}\beta_{i}p_{i}$.
The following holds.
\begin{enumerate}
\item There is a uniform $r_{s}>0$ such that for $r<r_{s}$, at each $p_{i}$,
the punctured geodesic disk $D_{r}(p_{i})-p_{i}\subset M$ is diffeomorphic
to a flat cone with angle $2\pi(1+\beta_{i})$. Each diffeomorphism
is given by an exponential map. 
\item There exists $d_{M}<r_{s}/2$ such that for any $q\in M-\cup_{i}D_{3r_{s}/4}(p_{i})$,
the ordinary exponential map $\exp_{q}:D_{d_{M}}(0)\subset T_{p}M\to M$
gives a diffeomorphism to its image.
\item The distance between any $2$ points in $M$ is realized by a smooth
shortest geodesic. 
\end{enumerate}
\label{lem:diffeo Lem 5.1}

\end{prop}

\begin{proof}
At each conic point $p_{i}$, choose a $\rho_{0,i}$ in Lemma \ref{lem:().local diff}
(1) in the appendix. We may choose $r_{s}=\min\{\rho_{0,i}\}$. The
diffeomorphism to a flat cone near $p_{i}$ is given in Proposition
\ref{lem:ttangent cone}. Since $d_{M}<\frac{r_{s}}{2}$, the existence
of a local normal coordinate near each regular point in $M-\cup_{i}D_{3r_{s}/4}(p_{i})$
is guaranteed, and the exponential map gives a local diffeomorphism.
(3) is proved in Proposition \ref{lem: geodesic for regular points}
in the appendix.
\end{proof}
Next, we discuss the covering space of $M$. If there are only $2$
singularities, we can join the two singularities $p,q$ with a smooth
curve $\gamma_{p,q}$. Then we can take two copies of $M-\gamma_{p,q}$
and glue them along the corresponding sides of the edges. This gives
a \textbf{double cover} of $M$ branched at $p$ and $q$. See Figure
\ref{fig:Double-branched-cover}.
\begin{figure}
\centering

\includegraphics[width=6cm]{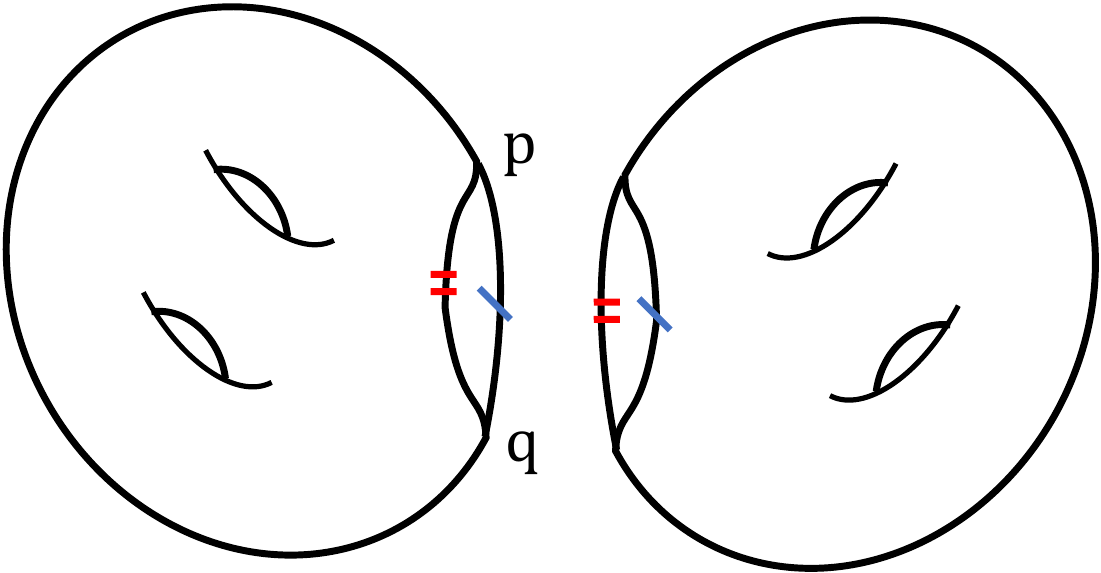}

\caption{\label{fig:Double-branched-cover}Double branched cover}
\end{figure}

Let $\mathcal{S}$ be a collection of conic points $p_{i}$ with cone
angles $\leq\pi$. If $\#\mathcal{S}$ is even, we can pair conic
points, for instance, $p,q$, in $\mathcal{S}$, and connect them
with a smooth curve $\gamma_{p,q}$. If $p,q,r,s\in\mathcal{S}$,
$(p,q)$ and $(r,s)$ are paired and $\gamma_{p,q}$ cross $\gamma_{r,s}$
at a regular point $a\in M$, let $U$ be a small open neighborhood
of $a$that does not contain any conic points and any other crossings.
To resolve the crossing at $a$, we remove the crossing $U\cap\left(\gamma_{p,q}\cup\gamma_{r,s}\right)$
and paste back $2$ smooth caps so that $(p,r)$ and $(q,s)$ are
connected. We then resolve all crossings one by one. As a result,
the points in $\mathcal{S}$ can be paired in a way such that the
connecting curves $\gamma_{p,q}$ and $\gamma_{r,s}$ for any two
pairs $(p,q)$ and $(r,s)$ do not cross each other. We take two
copies of $M$ and obtain a double branched cover $\tilde{M}$ by
first cutting along each connecting curve $\gamma_{p,q}$ and then
gluing each side to another copy. $\tilde{M}$ will branch exactly
at the points in $\mathcal{S}$. 

If $\#\mathcal{S}$ is odd, we can add a regular point $q$ on $M$
to $\mathcal{S}$ which we viewed as a ``singularity'' with cone
angle $2\pi$, we then make the branched double cover $\tilde{M}$
like before. The resulting space will have a singularity at $q$ with
cone angle $4\pi$. The following lemma enables us to pick a regular
point away from the curve shortening flow. 
\begin{lem}
Suppose that $X(x,t)$ exists for $t\in[0,T)$. Then, there is a point
$q\in M-\{p_{i}\}$ such that $\exists\epsilon,\delta>0$ for $t\in[T-\epsilon,T)$,
$B_{\delta}(q)\cap\{X(x,t):\forall x\in S^{1}\}=\varnothing$. \label{lem: regular point away from csf}
\end{lem}

\begin{proof}
We argue by contradiction. Suppose that for any $q\in M-\{p_{i}\}$,
there is a sequence of $\epsilon_{i}$ goes to 0 such that $t_{i}=T-\epsilon_{i}$
and $B_{\epsilon_{i}}(q)\cap X(\cdot,t_{i})\not=\varnothing.$ This,
however, means that $\limsup_{t\to T}L(t)=\infty$, which is impossible.
\end{proof}
Suppose that the covering map is given by $\boldsymbol{\pi}:\tilde{M}\to M$.
We can equip $\tilde{M}$ with the pull-back metric. We call $\tilde{O}\in\tilde{M}$
a \textbf{doubled conic point} if $\boldsymbol{\pi}(\tilde{O})=O\in\mathcal{S}$.
Suppose that the cone angle at $O$ is $2(1+\beta)\pi$. Then, we
can find a holomorphic coordinate $\xi$ near $O$, and a holomorphic
coordinate $z$ near $\tilde{O}$ such that $\xi(O)=0,z(O)=0$, and
\[
\boldsymbol{\pi}:z\mapsto z^{2}.
\]
We see immediately that the pull-back metric is
\[
\boldsymbol{\pi}^{*}g=|z^{2}|^{2\beta}e^{2h(z^{2})}|dz^{2}|^{2}=4|z|^{2(1+2\beta)}e^{2h(z^{2})}|dz|^{2}.
\]
Thus, the cone angle at $\tilde{O}$ is $2(2+2\beta)\pi$. Let $U$
be a disk neighborhood of $O$. If $h(z)\in C^{2,\alpha}(U,O\beta)$,
then $h(z^{2})\in C^{2,\alpha}(U,O(2\beta+1))$. Moreover, if $\beta\geq-\frac{1}{2}$,
then $h(z^{2})\in C^{2,\alpha}(U)$ in the regular H$\mathbf{\ddot{o}}$lder
space. In particular, if $\beta=-\frac{1}{2}$, then the pull-back
metric $\boldsymbol{\pi}^{*}g$ is a $C^{2,\alpha}$ metric at $\tilde{O}$.
To summarize, we can construct a double cover $\tilde{M}$ of $M$
branched over $\mathcal{S}$ or $\mathcal{S}\cup\{pt\}$ in the sense
of the following Proposition.
\begin{prop}
\label{prop:construction of the branched double cover}Suppose that
$(M,g,D)$ is a conic Riemann surface and $D=\sum_{i}\beta_{i}p_{i}.$
Let $k$ be the number of conic points of $M$. Let $\mathcal{S}$
be the collection of conic points $p_{i}$ with cone angles small
than $\pi$. Then there is generalized Riemann surface $(\tilde{M},\tilde{g},\tilde{D})$
such that
\begin{enumerate}
\item There exists a branched double cover $\boldsymbol{\mathbf{\pi}}:\tilde{M}\to M$
which is locally isometric away from the branched points.
\item If $X(x,t)$ for $t\in[0,T)$ is a CSF on $M$, then the lifting $\tilde{X}$
on $\tilde{M}$ is also a CSF.
\item If $\#\mathcal{S}$ is even, then $\boldsymbol{\pi}$ is branched
exactly at $\mathcal{S}$. $\tilde{M}$ is a conic Riemann surface
with at most $2k-\#\mathcal{S}$ conic points and 
\[
\tilde{D}=\sum_{p_{i}\in\mathcal{S}}(2\beta_{i}+1)p_{i}+\sum_{p_{i}\not\in\mathcal{S}}\beta_{i}\pi^{-1}(p_{i}).
\]
\item If $\#\mathcal{S}$ is odd, then $\boldsymbol{\pi}$ is branched at
$\mathcal{S}\cup\{\hat{p}\}$ for a regular point $\hat{p}$ of $M$.
$\tilde{M}$ has at most $2k-\#\mathcal{S}+1$ singular points and
\[
\tilde{D}=\sum_{p_{i}\in\mathcal{S}}(2\beta_{i}+1)p_{i}+\hat{p}+\sum_{p_{i}\not\in\mathcal{S}}\beta_{i}\pi^{-1}(p_{i}).
\]
In this case, we can pick $\hat{p}$ away from a particular CSF on
$M$ by Lemma \ref{lem: regular point away from csf}. 
\end{enumerate}
\end{prop}

\section{CSF on conic surfaces}

We study CSF on conic surfaces in this section. First, we rewrite
Theorem \ref{thm:Suppose-that-first theorem} in the following form.
\begin{thm}
Let $(M,g)$ be a conic Riemann surface with divisor $D=\sum_{i=1}^{k}\beta_{i}p_{i}$.
Suppose that $\mathcal{S}$ contains conic singularities with cone
angles $\leq\pi$. Let $\Gamma_{t}$ be a curve shortening flow parametrized
by $X(x,t)$. Suppose that $X(x,t)$ exists for $(x,t)\in S^{1}\times[0,T)$,
$T<\infty$. Let $\mathbf{d}(\cdot,\cdot)$ be the distance function
on $M$. Then, either $\mathbf{d}(\Gamma_{t},p)>c>0$ for all $t\in[0,T)$
and all $p\in\mathcal{S}$, or $\mathbf{d}(\Gamma_{t},p)\to0$ uniformly
for some $p\in\mathcal{S}$ as $t\to T$. \label{thm:restated thm1}
\end{thm}

To prove this Theorem, we use comparison function $R(t)$ on the double
branched cover $\tilde{M}$ described in Proposition \ref{prop:construction of the branched double cover}.
We will prove the boundedness of $R(t)$ with an additional assumption
that the total length $L(t)$ remains positive.
\begin{thm}
Let $\tilde{M}$ be a double cover of $M$ branched at $\mathcal{S}$.
Suppose that $X(x,t)$ exists for $t\in[0,T)$ and $L(t)>c>0$ for
all $t\in[0,T)$. Let $\tilde{X}(x,t)$ be the lift of $X(x,t)$ on
$\tilde{M}$. Let
\[
R(t)=\sup_{x,y}e^{-Kt}\frac{\tilde{L}(t)}{\pi d_{\tilde{M}}(x,y,t)}\sin\frac{\pi\tilde{l}(x,y,t)}{\tilde{L}(t)},
\]
where $\tilde{L}(t)=2L(t)$ is the total length of $\tilde{X}(\cdot,t)$,
$\tilde{l}$ is the arc-length on $\tilde{X}(\cdot,t)$, and $d_{\tilde{M}}(x,y,t)$
is the distance on $\tilde{M}$ between $\tilde{X}(x,t)$ and $\tilde{X}(y,t)$.
Then, there exists a constant $K=K(M,c)$ such that $R(t)$ is bounded
by a universal constant depending on $R(0)$ and $M$. \label{thm:conic case R(t) bounded}
\end{thm}

Let us postpone the proof of Theorem \ref{thm:conic case R(t) bounded}
and give a quick proof of Theorem \ref{thm:restated thm1}.
\begin{proof}
[Proof of Theorem \ref{thm:restated thm1} assuming Theorem \ref{thm:conic case R(t) bounded}]Let
$\tilde{M}$ be the double cover of $M$ branched at $\mathcal{S}$.
As $t\to T$, if ${\rm \mathbf{d}}_{M}(X(\cdot,t),p)\to0$ for some
$p\in\mathcal{S}$. Then, $\mathbf{d}_{\tilde{M}}(\tilde{X}(\cdot,t),\boldsymbol{\pi}^{-1}(p))\to0$
as well. Here, $\mathbf{d}_{M}$ denote the distance function on $M$
and $\mathbf{d_{\tilde{M}}}$ denote the distance function on $\tilde{M}.$

If $\tilde{X}$ does not shrink to $p$, then $\lim_{t\to T}\tilde{L}(t)\geq c>0$.
Let $\{(x_{i},t_{i})\}_{i=1}^{\infty}$ be a sequence such that $\tilde{X}(x_{i},t_{i})\to p$.
Then, $\tilde{X}(x_{i}+\frac{\tilde{L}}{2},t_{i})\to p$ by the symmetry
of the lifting. Now, evaluating at $(x_{i},t_{i})$ and $(x_{i}+\frac{\tilde{L}}{2},t_{i})$,
we have
\[
R(t_{i})\geq e^{-Kt_{i}}\frac{\tilde{L}(t_{i})}{\pi d_{\tilde{M}}(x_{i},x_{i}+\frac{\tilde{L}}{2},t_{i})}\sin\left(\frac{\pi}{2}\right)\geq e^{-KT}\frac{c}{\pi d_{\tilde{M}}(x_{i},x_{i}+\frac{\tilde{L}}{2},t_{i})},
\]
where $d_{\tilde{M}}(x,y,t)=\mathbf{d}_{\tilde{M}}(X(x,t),X(y,t))$.
The right hand side blows up as $i\to\infty$ because 
\[
d_{\tilde{M}}(x_{i},x_{i}+\frac{\tilde{L}}{2},t_{i})\leq2\mathbf{d}_{\tilde{M}}(\tilde{X}(x_{i},t_{i}),\boldsymbol{\pi}^{-1}(p))\to0.
\]
However, $R(t)$ is bounded, so we have reached a contradiction. 
\end{proof}
Now, we only need to prove the boundedness of $R(t)$. In order to
apply the argument in Theorem \ref{thm:R(t) is boundeddd}, we have
to take the derivatives of the distance function. For smooth surfaces,
by choosing a big upper-bound for $R(t)$, we can always restrict
ourselves inside a small disk in order to stay away from the cut locus.
However, near conic singularities, cut locus can not be avoided. 

Let $O$ be a conic point and let $A$ be a point in a disk neighborhood
$U$ centered at $O$ such that ${\rm diam}(U)<\min\{r_{s},\pi/\sqrt{K_{M}}\}$
where $r_{s}$ is from Proposition \ref{lem:diffeo Lem 5.1}. We denote
the cut locus of $A$ as $\mathcal{C}_{A}$. By Rauch comparison principle,
$\mathcal{C}_{A}\cap U$ does not contain conjugate points if ${\rm diam}(U)<\frac{\pi}{\sqrt{K_{M}}}$
. Therefore, 
\[
\mathcal{C}_{A}\cap U=\{B\in U:\exists{\rm geodesics}\ \gamma_{1}\not=\gamma_{2}\ {\rm realizing\ distance\ }\mathbf{d}(A,B)\}.
\]
If $B\in\mathcal{C}_{A}\cap U$, there are $2$ shortest geodesics
connecting $A$ and $B$. In fact, suppose that $3$ or more shortest
geodesics connect $A$ and $B$. $2$ of the geodesics, denoted as
$\gamma_{1},\gamma_{2}:[0,\mathbf{d}(A,B)]\to U$, will bound a simply-connected
bigon $\Omega_{A,B}\subset U$ which contains no conic points. Since
${\rm diam}(\Omega_{A,B})<\frac{\pi}{\sqrt{K_{M}}},$ and the Gaussian
curvature $\mathcal{K}<K_{M}$, we can apply Rauch comparison principle
in $\Omega_{A,B}$ to show $\mathbf{d}(\gamma_{1}(t),\gamma_{2}(t))>0$.
It contradicts to the fact that $\gamma_{i}(\mathbf{d}(A,B))=B$,
$i=1,2$.

We will prove the following fact: If $Z_{N}(\cdot,\cdot,t)$ achieves
local minimum at $(x,y)$ such that $X(x,t)$ is inside the cut locus
of $X(y,t)$ near a conic point, then $Z_{N}$ can be replaced by
some smooth functions which also achieve minimum at the same location.
\begin{lem}
\label{lem:cut-locusnot-in a-conic_neighborhood}Let $D_{r}(O)$ be
a disk neighborhood of a conic point $O$. Suppose that $Z_{N}(\cdot,\cdot,t)$
achieves local minimum at $\bar{x},\bar{y}$ and $A=X(\bar{x},t)$,
$B=X(\bar{y},t)$ are both in $D_{r}(O)\backslash\{O\}$. Then, $Z_{N}(\cdot,\cdot,t)$
is $C^{1}$ in a neighborhood of $(\bar{x},\bar{y})$. Moreover, if
$B\in\mathcal{C}_{A}$ , then there exist two smooth functions $Z_{N}^{i}(\cdot,\cdot,t)$,
$i=1,2$ such that $Z_{N}(x,y,t)=\min\{Z_{N}^{1}(x,y,t),Z_{N}^{2}(x,y,t)\}$
and $Z_{N}^{i}(\cdot,\cdot,t)$ both achieve local minimum at $(\bar{x},\bar{y})$.
\end{lem}

\begin{proof}
Let $f_{1}(y)$, $f_{2}(y)$ be two smooth functions, and $f(y):=\min\{f_{1}(y),f_{2}(y)\}$.
Then, $f(y)$ is smooth away from the set $\mathcal{C}=\{y:f_{1}(y)=f_{2}(y)\}$.
If $f(y)$ achieves a local minimum at $y_{0}\in\mathcal{C}$ then
we claim that $y_{0}$ is a local minimum of both $f_{1}$ and $f_{2}$.
Otherwise, we may assume that there exists a sequence $y_{i}\to y_{0}$,
as $i\to\infty$ such that $f_{1}(y_{i})<f_{1}(y_{0})$. Then 
\[
f(y_{i})\leq f_{1}(y_{i})<f_{1}(y_{0})=f(y_{0}).
\]
This contradicts to the fact that $f$ achieves a local minimum at
$y_{0}$. So we have proved the claim. Moreover, we see that $f_{1}'(y)=f_{2}'(y)=0$.

We assume that $B\in\mathcal{C}_{A}$. Otherwise, the conclusion is
obvious. Let $\mathbf{d}(\cdot,\cdot)$ be the distance function.
Since $B\in\mathcal{C}_{A}$, there exist two geodesics $\gamma_{1},\gamma_{2}$
connecting $A,B$. We denote $d_{1}(A,B)={\rm Length}(\gamma_{1})$
and $d_{2}(A,B)={\rm Length}(\gamma_{2})$. We can extend each $\gamma_{i}$
to be a family of geodesics connecting $A$ and any point $Q$ near
$B$. Thus, we can extend $d_{1}(A,Q)$ and $d_{2}(A,Q)$ smoothly
for $Q$ in a neighborhood of $B$ such that
\begin{figure}
\includegraphics[width=5cm]{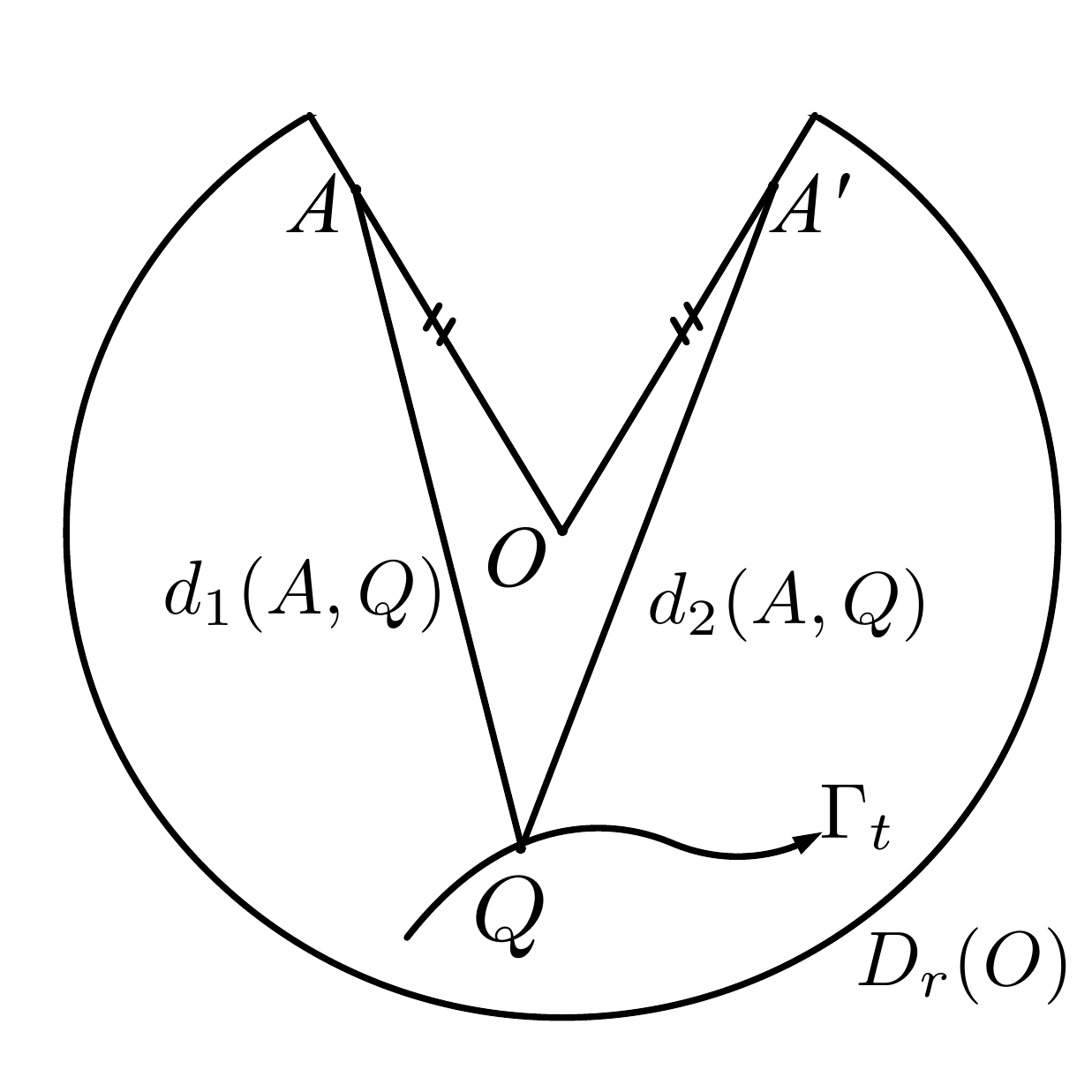}\caption{Distance function ($A,A'$ are the same point)}
\label{figure,Distancefunction}
\end{figure}
\[
\mathbf{d}(A,Q)=\min\{d_{1}(A,Q),d_{2}(A,Q)\}.
\]
See Figure \ref{figure,Distancefunction}. For $y$ in a neighborhood
of $\bar{y}$, we define
\[
Z_{N}^{i}(\bar{x},y,t):=Nd_{i}(A,X(y,t))+\frac{L(t)}{\pi}\sin\left(\frac{\pi l(x,y,t)}{L(t)}\right)e^{-Kt}.
\]
Then, each $Z_{N}^{i}(\bar{x},\cdot,t)$ is a function smoothly defined
in a neighborhood of $\bar{y}$ and $Z_{N}(\bar{x},y,t)=\min\{Z_{N}^{1},Z_{N}^{2}\}$.
Now, $B\in\mathcal{C}_{A}$ implies that $d_{1}(A,B)=d_{2}(A,B)$,
and $Z_{N}^{1}(\bar{x},\bar{y},t)=Z_{N}^{2}(\bar{x},\bar{y},t)$.
By the previous argument, we see that if $Z_{N}(\bar{x},\cdot,t)$
achieves local minimum at $\bar{y}$, then $Z_{N}^{i}(\bar{x},\cdot,t)$
also achieves local minimum at $\bar{y}$ and $\left(\frac{\partial Z_{N}}{\partial y}\right)(\bar{x},\bar{y},t)=0$
exists. By symmetry, $A\in\mathcal{C}_{B}$ and we can apply the same
argument for $Z_{N}(\cdot,\bar{y},t)$. Thus, we have finished the
proof.
\end{proof}

\begin{proof}
[Proof of Theorem \ref{thm:conic case R(t) bounded}] Lift $X(x,t)$
to be a CSF $\tilde{X}(x,t)$ on a double branched cover $\tilde{M}$
given in Proposition \ref{prop:construction of the branched double cover}.
Let $d_{0}=\min\{d_{M},\delta_{0}\}$, where $d_{M}$ is defined in
Corollary \ref{lem:diffeo Lem 5.1} and $\delta_{0}$ is defined in
Lemma \ref{lem:Jacobian lemma}. We pick $N>\max\{\frac{\tilde{L}(0)}{\pi d_{0}},R(0)\}$.
If $Z_{N}(x,y,\bar{t})$ allows local variation at $\bar{x},\bar{y}$
, then the same argument in Theorem \ref{thm:R(t) is boundeddd} can
be applied. We require $L(t)>c>0$ to avoid subcase 2b and use (\ref{eq:R(t) bounded case 2-1}),
where the condition (\ref{eq:K condition case 2}) on $K$ is replaced
by $K>4C_{1}+\frac{4\pi^{2}}{c^{2}}$. If $X(\bar{x},\bar{t})$ is
in the cut locus of $X(\bar{y},\bar{t})$, then we may assume that
$X(\bar{x},\bar{t})$ and $X(\bar{y},\bar{t})$ is in a neighborhood
of a conic singularity $O$ by the choice of $N$. Then, by Lemma
\ref{lem:cut-locusnot-in a-conic_neighborhood}, 
\[
Z_{N}(x,y,\bar{t})=\min\{Z_{N}^{1}(x,y,\bar{t}),Z_{N}^{2}(x,y,\bar{t})\},
\]
for $(x,y)$ in a neighborhood of $(\bar{x},\bar{y})$, $Z_{N}^{1}(\bar{x},\bar{y},\bar{t})=Z_{N}^{2}(\bar{x},\bar{y},\bar{t})$,
and both $Z_{N}^{i}(\cdot,\cdot,\bar{t})$ achieve local minimum at
$(\bar{x},\bar{y}$). Then, we can apply the same arguments in Theorem
\ref{thm:R(t) is boundeddd} to either $Z_{N}^{1}(x,y,t)$ or $Z_{N}^{2}(x,y,t)$
which also yields contradictions.
\end{proof}
\appendix

\section{Coordinates and geodesics on conic surfaces }

In this appendix, we give the proof of the results in Proposition
\ref{lem:diffeo Lem 5.1}. We first prove the existence of an exponential
map at each conic point. In \cite{troyanov1990coordonnees}, Troyanov
shows the existence of a geodesic polar coordinate near a conic point
assuming a Gaussian curvature bound. Since we assume $h\in C^{\infty}(M-\{p_{i}\})\cap C^{2,\alpha}(M,D)$,
the regularity of the polar coordinate is much easier to achieve.
\begin{lem}
(\cite{troyanov1990coordonnees}). \label{lem:().local diff}Let $U$
be a neighborhood of conic point $p$ such that the divisor $D|_{U}=\beta p$,
and the metric $g|_{U}$ is in the form of (\ref{eq:conic metric def})
for some $h\in C^{2,\alpha}(M,D)$. The following holds. 
\begin{enumerate}
\item There is a $\rho_{0}>0$, an open neighborhood $V\subset U$ of $p$,
and a map $\varphi:[0,\rho_{0})\times S^{1}\to V$ such that such
that $\varphi(0,\theta)=p$, and $\varphi|_{(0,\rho_{0})\times S^{1}}\to V-\{p\}$
is smoothly diffeomorphic.
\item The pull-back metric 
\begin{equation}
\varphi^{*}g=d\rho^{2}+\phi^{2}(\rho,\theta)d\theta^{2},\label{eq:polar coordinate}
\end{equation}
where $\phi^{2}\in C^{2,\alpha}(D_{\rho_{0}}(p),\beta p)$, and $\lim_{\rho\to0}\frac{\phi}{\rho}=1+\beta$.
\end{enumerate}
We call $\varphi:[0,\rho_{0})\times S^{1}$ a geodesic polar coordinate
near $p$.

\end{lem}

\begin{proof}
We provide a proof for readers' convenience. Suppose that $g_{0}$
is a smooth background metric on $M$. Let $(r,\theta)$ be a geodesic
polar coordinate on a geodesic disk $D_{r_{0}}^{g_{0}}(p)\subset(M,g_{0})$,
such that $r(p)=0$, and $g_{0}=dr^{2}+\psi(r,\theta)^{2}d\theta^{2}$.
Since $g_{0}$ is smooth, $\lim_{r\to0}\psi/r=\lim_{r\to0}\partial_{r}\psi=1.$
We may write $g=r^{2\beta}e^{2h(r,\theta)}(dr^{2}+\psi(r,\theta)^{2}d\theta^{2}).$
Let $\rho(r)=\int_{0}^{r}s^{\beta}e^{h(s,\theta)}ds$. We obtain a
function $\rho$ in $D_{r_{0}}^{g_{0}}(p)$ strictly increasing in
$r$and smooth in $(r,\theta)$ for $r>0$.  We define
\[
\varphi(\rho,\theta)=(r,\theta).
\]
Note, $h$ is continuous and $\rho_{0}=\inf_{\partial D_{r_{0}}^{g_{0}}}\rho>0$.
Let $V=\varphi([0,\rho_{0})\times S^{1})$. Then, $\varphi(0,\theta)=p$
and $\varphi$ is a smooth diffeomorphic map from $(0,\rho_{0})\times S^{1}$
to $V-\{p\}$. Notice that $\varphi^{*}g=d\rho^{2}+\phi(\rho,\theta)^{2}d\theta^{2},$
for $\phi(\rho,\theta)=r^{\beta}e^{h}\psi(r,\theta).$
\begin{align*}
\lim_{\rho\to0}\phi/\rho & =\lim_{\rho\to0}\frac{\partial\phi}{\d r}/\frac{\d\rho}{\d r}=\beta+\lim_{r\to0}\left(\psi\frac{\partial h}{\d r}+\frac{\partial\psi}{\d r}\right)\\
 & =\beta+1.
\end{align*}
Here $\lim_{r\to0}\psi\frac{\partial h}{\d r}=0$ since $h\in C^{2,\alpha}(M,D)$.
Moreover, function $\psi(r,\theta)^{2}r^{2\beta}$ belongs to $C^{\infty}(U,\beta p)$.
Hence, 
\[
\phi^{2}(\rho,\theta)=e^{2h}\psi^{2}r^{2\beta}\in C^{2,\alpha}(U,\beta p).
\]
We have proved (1), (2).
\end{proof}
The next Lemma shows that the tangent cone at a conic point $p$ is
uniquely determined and isometric to a flat cone. A tangent cone at
$p$ is a pointed-Gromov-Hausdorff limiting space of $(M,\lambda_{i}^{-2}g,p)$
for a positive sequence $\lambda_{i}\to0$. See \cite{cheeger1997structure}.
Then, the exponential map at $p$ is well defined once we replace
the tangent plane at $p$ by the tangent cone space at $p$.
\begin{prop}
Assume the condition in Lemma \ref{lem:().local diff} and suppose
that $\varphi:[0,\rho_{0})\times S^{1}\to V$ gives a polar coordinate
near $p$. Let $T_{p}M$ be the tangent cone at the conic point $p$.
The following statements hold. \label{lem:ttangent cone}
\begin{enumerate}
\item The tangent cone $T_{p}M$ at $p$ is isometric to the flat cone $C_{(2+2\beta)\pi}$
with cone angle $(2+2\beta)\pi$. 
\item The exponential map $\exp_{p}(v)$ from the tangent cone to $V$ is
well-defined and diffeomorphic in a neighborhood of the cone tip..
\end{enumerate}
\end{prop}

\begin{proof}
For a positive sequence $\lambda_{i}\to0$, consider the pointed-Gromov-Hausdorff
limiting space of $(M,\lambda_{i}^{-2}g,p)$. In the polar coordinate,
for any $(\rho,\theta)\in(0,\infty)\times S^{1}$, let $\varphi_{i}(\rho,\theta):=\varphi(\lambda_{i}\rho,\theta)$.
If $\lambda_{i}<\frac{\rho_{0}}{\rho}$, 
\[
\varphi_{i}(\rho,\theta)=\varphi(\lambda_{i}\rho,\theta)\in V-\{p\}.
\]
Let $\phi_{i}(\rho)=\lambda_{i}^{-1}\phi(\lambda_{i}\rho).$ Then,
\[
\lim_{i\to\infty}\phi_{i}(\rho)=\rho\lim_{i\to\infty}\frac{\phi(\lambda_{i}\rho,\theta)}{\lambda_{i}\rho}=(1+\beta)\rho,
\]
by Lemma \ref{lem:().local diff} (2). Denote $\xi=\left((1+\beta)\rho\right)^{\frac{1}{1+\beta}}e^{\sqrt{-1}\theta}$.
\begin{align}
\lambda_{i}^{-2}\varphi_{i}^{*}g & \to d\rho^{2}+(1+\beta)^{2}\rho^{2}d\theta^{2}\label{eq:limiting space tangent cone}\\
 & =|\xi|^{2\beta}|d\xi|^{2}.\nonumber 
\end{align}
Hence, the tangent cone $T_{p}M$ at $p$ is isometric to the flat
cone $C_{\left(2+2\beta\right)\pi}$ defined by (\ref{eq:flat cone}).
The limit in (\ref{eq:limiting space tangent cone}) is independent
of the sequence $\lambda_{i}$ which implies the uniqueness of the
tangent cone. We have proved (1).

Let $v\in T_{p}M$. We may identify $v$ with $\left((1+\beta)\rho_{v}\right)^{\frac{1}{1+\beta}}e^{\sqrt{-1}\theta_{v}}$
in $\mathbb{C}$. By Lemma 1.4 in \cite{troyanov1990coordonnees},
there is a neighborhood $U_{1}$ of $p$ such that for any $q\in U_{1}$,
there is a unique geodesic $\gamma_{p,q}$ connecting $p$ and $q$.
Moreover, if $q=\varphi(\rho_{1},\theta_{1})$, then $\gamma_{p,q}=\{\varphi(t\rho_{1},\theta_{1}):t\in[0,1]\}$,
and the distance between $p,q$ is $\mathbf{d}(p,q)=\rho_{1}$. Thus,
if $\rho_{v}$ is small, the exponential map $\exp_{p}(v):=\varphi(\rho_{v},\theta_{v})$
is well defined and corresponds to the shortest geodesic of length
$\rho_{v}$ starting at $p$ in the direction of $v$. $\exp_{p}$
is a local diffeomorphism near the cone tip by Lemma \ref{lem:().local diff}
(1). 
\end{proof}
\begin{rem}
Let $\xi_{1},\xi_{2}\in T_{p}M$. Since $T_{p}M$ is isometric to
$C_{\left(2+2\beta\right)\pi}$, we may identify $\xi_{1}=\left((1+\beta)\rho_{1}\right)^{\frac{1}{1+\beta}}e^{\sqrt{-1}\theta_{1}}$,
and $\xi_{2}=\left((1+\beta)\rho_{2}\right)^{\frac{1}{1+\beta}}e^{\sqrt{-1}\theta_{2}}$,
where $0\leq\theta_{1},\theta_{2}<2\pi$. The angle between $\xi_{1}$
and $\xi_{2}$ is $\delta\theta=(1+\beta)\min\{|\theta_{2}-\theta_{1}|,|2\pi+\theta_{1}-\theta_{2}|\}$.
Then, the distance between $\xi_{1}$ and $\xi_{2}$ is given by 
\begin{equation}
\mathbf{d}_{T_{p}M}(\xi_{1},\xi_{2})^{2}=\rho_{1}^{2}+\rho_{2}^{2}-2\rho_{1}\rho_{2}\cos\left(\delta\theta\right).\label{eq: distance on flat cone}
\end{equation}
\end{rem}

Next, we show the existence of shortest geodesics for any pairs of
points on $M$.
\begin{prop}
Let $q_{1},q_{2}$ be $2$ points on $M$. There is a shortest smooth
geodesic realizing the distance between $q_{1}$ and $q_{2}$.\label{lem: geodesic for regular points}
\end{prop}

\begin{proof}
$(M,g,D)$ is a locally compact length space. It is also complete
as a metric space. By Hopf-Rinow Theorem (\cite{burago2001course},
Section 2.5), the distance $l=\mathbf{d}(q_{1},q_{2})$ is realized
by a piecewise geodesic $\gamma:[0,l]\to M$. The differentiable parts
of $\gamma$ in the regular part of $M$ have to be smooth geodesics.
If the interior of $\gamma$ does not contain any conic points, $\gamma$
is a smooth geodesic and we are done. If not, we argue by contradiction. 

Let $p=\gamma(t_{1})$ be a conic point for some $t_{1}\in(0,l)$.
Let $V$ be a neighborhood of $p$ in Lemma \ref{lem:().local diff}
and $D|_{V}=\beta p$. Then, $\gamma\cap V$ contains at least $2$
geodesic segments connecting $p$. By Proposition \ref{lem:ttangent cone},
we may assume that for some $0\leq\theta_{1},\theta_{2}<2\pi$,
\[
\gamma(t)=\begin{cases}
\varphi(t_{1}-t,\theta_{1}), & t\in[t_{0},t_{1}),\\
\varphi(t-t_{1},\theta_{2}), & t\in[t_{1},t_{2}).
\end{cases}
\]
Let $\tau=\min\{t_{1}-t_{0},t_{2}-t_{1}\}/2$, $\delta\theta=(1+\beta)\min\{|\theta_{2}-\theta_{1}|,|2\pi+\theta_{1}-\theta_{2}|\}$.
Note $\delta\theta<\pi$, since $0<1+\beta<1$. Let $\{\lambda_{i}\}$
be a positive sequence which goes to $0$. By the minimizing property
of $\gamma$, 
\begin{equation}
\mathbf{d}(\varphi(\lambda_{i}\tau,\theta_{1}),\varphi(\lambda_{i}\tau,\theta_{2}))=2\lambda_{i}\tau.\label{eq:contradiction to local minimizing}
\end{equation}
By (1) in Proposition \ref{lem:ttangent cone} and (\ref{eq: distance on flat cone}),
\begin{align}
\lim_{i\to\infty}\lambda_{i}^{-2}\mathbf{d}^{2}(\varphi(\lambda_{i}\tau,\theta_{1}),\varphi(\lambda_{i}\tau,\theta_{2})) & =2\tau^{2}\left(1-\cos\delta\theta\right).\label{eq:limiting of distance dunction}
\end{align}
We can choose $0<\epsilon<\sqrt{\tau^{2}(1+\cos\delta\theta)}$. By
(\ref{eq:limiting of distance dunction}), $\exists N>0$ such that
if $i>N$, 
\[
|\mathbf{d}^{2}(\varphi(\lambda_{i}\tau,\theta_{1}),\varphi(\lambda_{i}\tau,\theta_{2}))-2\lambda_{i}^{2}\tau^{2}(1-\cos\delta\theta)|<\epsilon^{2}\lambda_{i}^{2}.
\]
Then,
\[
\mathbf{d}^{2}(\varphi(\lambda_{i}\tau,\theta_{1}),\varphi(\lambda_{i}\tau,\theta_{2}))<\lambda_{i}^{2}\tau^{2}\left(3-\cos\delta\theta\right)<4\lambda_{i}^{2}\tau^{2},
\]
which contradicts to (\ref{eq:contradiction to local minimizing}).
We have proved the proposition.
\end{proof}

\bibliographystyle{alpha}
\bibliography{curveshorteningflow}

\end{document}